\documentclass[12pt,svgnames]{article}
\usepackage{amsmath}
\usepackage{amssymb}
\usepackage{amsthm}
\usepackage{wasysym}
\usepackage{enumerate}
\usepackage{dsfont}
\usepackage{stmaryrd}
\usepackage{comment}
\usepackage{tikz}
\usepackage{tikz-cd}
\usepackage{extarrows}
\usepackage{mathrsfs}
\usepackage{quiver}
\usepackage{enumitem}
\usepackage{fullpage}
\usepackage{hyperref}
\usepackage{titlesec}
\usepackage[all]{xy}
\hypersetup{
    linkbordercolor = {PaleTurquoise},
    citebordercolor = {PaleGreen},
}

\titleformat{\section}{\normalsize\bfseries}{\thesection}{1em}{}
\titleformat{\subsection}{\normalsize\bfseries}{\thesubsection}{1em}{}

\numberwithin{equation}{subsection}

\theoremstyle{definition}

\newtheorem{defn}[subsection]{Definition}
\newtheorem{para}[subsection]{}

\newtheorem{rmk}[subsection]{Remark}
\newtheorem{assumption}[subsection]{Assumption}

\newtheorem*{assumption*}{Assumption}

\newtheorem{defn_sub}[subsubsection]{Definition}
\newtheorem{para_sub}[subsubsection]{}
\newtheorem{egg_sub}[subsubsection]{Example}
\newtheorem{rmk_sub}[subsubsection]{Remark}

\theoremstyle{plain}

\newtheorem{prop}[subsection]{Proposition}
\newtheorem{theo}[subsection]{Theorem}
\newtheorem{lem}[subsection]{Lemma}
\newtheorem{cor}[subsection]{Corollary}

\newtheorem*{claim*}{Claim}
\newtheorem*{just*}{Justification}
\newtheorem*{lem*}{Lemma}
\newtheorem*{prop*}{Proposition}

\newtheorem{prop_sub}[subsubsection]{Proposition}
\newtheorem{theo_sub}[subsubsection]{Theorem}
\newtheorem{lem_sub}[subsubsection]{Lemma}
\newtheorem{cor_sub}[subsubsection]{Corollary}

\newcommand{\J}{\mathscr{J}}
\newcommand{\V}{\mathscr{V}}

\newcommand{\ob}{\mathop{\mathsf{ob}}}

\newcommand{\tensor}{\otimes}
\newcommand{\Mnd}{\mathsf{Mnd}}

\newcommand{\Lan}{\mathsf{Lan}}
\newcommand{\Cat}{\text{-}\mathsf{Cat}}

\newcommand{\Cocts}{\text{-}\mathsf{Cocts}}
\newcommand{\underJ}{\kern -0.5ex \mathscr{J}}
\newcommand{\Set}{\mathsf{Set}}

\newcommand{\op}{\mathsf{op}}
\newcommand{\Th}{\mathsf{Th}}
\newcommand{\T}{\mathbb{T}}
\newcommand{\C}{\mathscr{C}}

\newcommand{\B}{\mathscr{B}}
\newcommand{\A}{\mathscr{A}}

\newcommand{\Mod}{\text{-}\mathsf{Mod}}

\newcommand{\y}{\mathsf{y}}

\newcommand{\Alg}{\text{-}\mathsf{Alg}}

\newcommand{\lambdabar}{\bar{\lambda}}

\newcommand{\M}{\mathscr{M}}
\newcommand{\limit}{\mathop{\mathsf{lim}}}

\newcommand{\Ran}{\mathsf{Ran}}

\newcommand{\FinCard}{\mathsf{FinCard}}
\newcommand{\N}{\mathbb{N}}

\newcommand{\Sem}{\mathop{\mathsf{Sem}}}

\newcommand{\colim}{\mathop{\mathsf{colim}}}
\newcommand{\G}{\mathscr{G}}
\newcommand{\Iso}{\mathsf{Iso}}

\newcommand{\D}{\mathscr{D}}
\newcommand{\bbD}{\mathbb{D}}

\newcommand{\scrT}{\mathscr{T}}

\newcommand{\Cts}{\text{-}\mathsf{Cts}}
\newcommand{\X}{\mathscr{X}}

\newcommand{\E}{\mathscr{E}}

\newcommand{\All}{\mathsf{All}}

\newcommand{\CAT}{\text{-}\mathsf{CAT}}
\newcommand{\ALG}{\mathsf{Alg}}

\newcommand{\scrK}{\mathscr{K}}

\newcommand{\Preth}{\mathsf{Preth}}

\newcommand{\Tract}{\text{-}\mathsf{Tract}}

\newcommand{\scrU}{\mathscr{U}}
\newcommand{\Str}{\mathop{\mathsf{Str}}}
\newcommand{\MONADIC}{\mathsf{Monadic}}
\newcommand{\calT}{\mathcal{T}}
\newcommand{\Ner}{\text{-}\mathsf{Ner}}

\newcommand{\SF}{\mathsf{SF}}
\newcommand{\Y}{\mathscr{Y}}

\newcommand{\bbS}{\mathbb{S}}

\newcommand{\sfa}{\mathsf{a}}

\newcommand{\sfE}{\mathsf{E}}

\newcommand{\SET}{\mathsf{SET}}

\newcommand{\sfM}{\mathsf{M}}
\newcommand{\sfm}{\mathsf{m}}
\newcommand{\sft}{\mathsf{t}}
\newcommand{\scrL}{\mathscr{L}}

\begin{document}

\title{\Large \textbf{Enriched structure--semantics adjunctions and monad--theory equivalences for subcategories of arities}}
\author{Rory B. B. Lucyshyn-Wright\let\thefootnote\relax\thanks{We acknowledge the support of the Natural Sciences and Engineering Research Council of Canada (NSERC), [funding reference numbers RGPIN-2019-05274, RGPAS-2019-00087, DGECR-2019-00273].  Cette recherche a \'et\'e financ\'ee par le Conseil de recherches en sciences naturelles et en g\'enie du Canada (CRSNG), [num\'eros de r\'ef\'erence RGPIN-2019-05274, RGPAS-2019-00087, DGECR-2019-00273].} \medskip \\ Jason Parker \medskip
\\
\small Brandon University, Brandon, Manitoba, Canada}
\date{}

\maketitle

\begin{abstract}
Lawvere's \textit{algebraic theories}, or \textit{Lawvere theories}, underpin a categorical approach to general algebra, and Lawvere's adjunction between \textit{semantics} and algebraic \textit{structure} leads to an equivalence between Lawvere theories and finitary \textit{monads} on the category of sets.  Several authors have transported these ideas to a variety of settings, including contexts of category theory enriched in a symmetric monoidal closed category.  In this paper, we develop a general axiomatic framework for enriched structure--semantics adjunctions and monad--theory equivalences for \textit{subcategories of arities}. Not only do we establish a simultaneous generalization of the monad--theory equivalences
previously developed in the settings of Lawvere (1963), Linton (1966), Dubuc (1970), Borceux-Day (1980), Power (1999), Nishizawa-Power (2009), Lack-Rosick\'y (2011), Lucyshyn-Wright (2016), and Bourke-Garner (2019), but also we establish
a structure–semantics theorem that generalizes those given in the first four of these works while applying also to the remaining five, for which such a result has not previously been developed. Furthermore, we employ our axiomatic framework to establish broad new classes of examples of enriched monad--theory equivalences
and structure--semantics adjunctions for subcategories of arities enriched in locally bounded closed categories, including various convenient closed categories that are relevant in topology and analysis and need not be locally presentable.
\end{abstract}

\section{Introduction}\label{sec:intro}

Syntactic presentations of categories of general algebraic structures (or \textit{algebras}) were delineated by Birkhoff \cite{Birkhoff} in terms of operations and equations, with the complication that a given syntactic presentation cannot be recovered as an isomorphism-invariant attribute of its category of algebras $\A$, even when $\A$ is equipped with its underlying-set functor $U:\A \rightarrow \Set$ and so is viewed as a \textit{category over $\Set$}.  Lawvere \cite{Law:PhD} decisively overcame this complication through the insight that the derived operations carried by the algebras in $\A$ are precisely natural transformations $U^n \Rightarrow U$ ($n \in \N$) and thus constitute the morphisms of a category $\Str U$, called the \textit{algebraic structure} of $U$ (or of $\A$), which is an equivalence-invariant attribute of any category $\A$ over $\Set$. Lawvere's decisive methodological advance was to axiomatize a notion of \textit{algebraic theory} (or \textit{Lawvere theory}) as a small category $\scrT$ whose objects are the finite powers $T^n$ of an object $T$, and to define the category $\scrT\Alg$ of \textit{$\scrT$-algebras} as the category of all functors $A:\scrT \rightarrow \Set$ that preserve finite products.  In Lawvere's terminology, the category $\scrT\Alg$ over $\Set$ is called the \textit{semantics} of $\scrT$. Each Lawvere theory $\scrT$ can then be recovered (up to isomorphism) as the algebraic structure of its semantics.  Up to equivalence, the categories of algebras of Lawvere theories are precisely Birkhoff's categories of algebras, since each of the latter is equivalent to the semantics of its algebraic structure.

Lawvere thus established a dual equivalence between the category of Lawvere theories and the category of \textit{algebraic} categories over $\Set$.  Calling a category over $\Set$ \textit{tractable} if its algebraic structure is small, Lawvere established an adjunction between the category of tractable categories over $\Set$ and the opposite of the category of Lawvere theories, in which the left adjoint sends each tractable category over $\Set$ to its algebraic structure, while the fully faithful right adjoint sends each Lawvere theory to its semantics.  The resulting \textit{structure--semantics adjunction} therefore provides a reflective embedding of (the opposite of the category of) Lawvere theories into tractable categories over $\Set$, whose essential image consists of the algebraic categories over $\Set$.  It has been known since the work of Linton \cite{Lintonequational} that these are equivalently the categories of algebras of finitary \textit{monads} on $\Set$, so that by employing Lawvere's structure--semantics adjunction together with its counterpart for monads, one obtains an equivalence between the categories of Lawvere theories and of finitary monads, the classic \textit{monad--theory equivalence} of finitary algebra.

Linton \cite{Lintonequational} also established a structure--semantics adjunction between the category of infinitarily tractable categories over $\Set$ and the opposite of the category of \emph{varietal} theories (i.e.~infinitary Lawvere theories), leading to an equivalence between varietal theories and arbitrary monads on $\Set$.  Dubuc \cite{Dubucsemantics} then generalized these results to the enriched setting by establishing, for a complete and well-powered symmetric monoidal closed category $\V$, a structure--semantics adjunction between the category of \emph{$\V$-tractable} $\V$-categories over $\V$ and the opposite of the category of \emph{$\V$-theories}, from which he deduced an equivalence between $\V$-theories and arbitrary $\V$-monads on $\V$. Borceux and Day \cite{BorceuxDay} later established a structure--semantics adjunction for $\V$-enriched Lawvere theories with natural number arities, for a symmetric monoidal closed \emph{$\pi$-category} $\V$. Power \cite{PowerLawvere} subsequently generalized the classic finitary monad--theory equivalence to the locally presentable enriched setting, by establishing an equivalence between (finitary) \emph{Lawvere $\V$-theories} and finitary $\V$-monads on $\V$ for a locally finitely presentable closed category $\V$. 

A theme of contemporary interest in enriched category theory is the study of classes of enriched monads and theories defined relative to a \emph{subcategory of arities}, i.e.~a full subcategory that is dense (in the enriched sense). This concept dates back to work of Linton \cite{Lintonoutline} and Diers \cite{Diers}, and was used to define the theories and monads with arities studied by Berger, Melli\`{e}s, and Weber \cite{BMW}. It also relates to the work \cite{LR} of Lack and Rosick\'{y}, wherein they establish a monad--theory equivalence between \emph{Lawvere $\Phi$-theories} and \emph{$\Phi$-accessible} $\V$-monads for a class of weights $\Phi$ satisfying their Axiom A. In \cite{EAT}, the first author studied enriched \emph{$\J$-theories} and \emph{$\J$-ary monads} for a \emph{system of arities} $\J \hookrightarrow \V$, which is a (possibly large) subcategory of arities that is closed under the monoidal structure. Generalizing the above monad--theory equivalences of Lawvere, Linton, Dubuc, and Power, the paper \cite{EAT} establishes a monad--theory equivalence for \emph{eleutheric} systems of arities in symmetric monoidal closed categories.    

Building on work by Power \cite{PowerLawvere}, Nishizawa-Power \cite{NishizawaPower}, and Berger-Melli\`{e}s-Weber \cite{BMW}, Bourke and Garner \cite{BourkeGarner} subsequently established a monad--theory equivalence for arbitrary small subcategories of arities in \emph{locally presentable} enriched categories. Specifically, given a small subcategory of arities $\J \hookrightarrow \C$ in a locally presentable $\V$-category $\C$ enriched in a locally presentable closed category $\V$, they showed that the category of $\J$-theories is equivalent to the category of \emph{$\J$-nervous} $\V$-monads on $\C$. 

Neither of the monad--theory equivalences of the first author \cite{EAT} or Bourke-Garner \cite{BourkeGarner} generalizes the other. While the first does not require that $\V$ be locally presentable or that the subcategory of arities $\J$ be small, it nevertheless requires that $\J$ be a \emph{system} of arities that is \emph{eleutheric}, and it also requires that $\C = \V$. And while the second does not require that $\C = \V$ or that $\J$ be eleutheric, it does require that $\C$ and $\V$ be locally presentable, and that $\J$ be small. Moreover, neither work provides a general treatment of enriched structure--semantics adjunctions that would specialize to those established by Lawvere, Linton, Dubuc, and Borceux-Day.     

Motivated by these considerations, in the present paper we develop a general axiomatic framework for studying enriched structure--semantics adjunctions and monad--theory equivalences for subcategories of arities, which generalizes all of the aforementioned results and also yields substantial new classes of examples. Given a $\V$-category $\C$ enriched in a symmetric monoidal closed category $\V$ that we generally only assume has equalizers, a \emph{subcategory of arities} is a (possibly large) full sub-$\V$-category $j : \J \hookrightarrow \C$ that is dense in the enriched sense. In particular, we do \emph{not} generally assume that $\J$ is small, or that $\C$ and $\V$ are (co)complete. A \emph{$\J$-pretheory} is a $\V$-category $\scrT$ equipped with an identity-on-objects $\V$-functor $\tau : \J^\op \to \scrT$, and a \emph{$\J$-theory} is a $\J$-pretheory satisfying a further condition involving \emph{$j$-nerves}. These concepts, for which we adopt the terminology of Bourke and Garner \cite{BourkeGarner}, originate with Linton \cite{Lintonoutline} and Diers \cite{Diers} in the unenriched setting, and were also employed by Nishizawa and Power \cite{NishizawaPower} for finitary enriched Lawvere theories. We base our study of structure--semantics adjunctions and monad--theory equivalences on axiomatic assumptions on the subcategory of arities related to the existence of \emph{free algebras}. Specifically, we say that a subcategory of arities $j : \J \hookrightarrow \C$ is \emph{amenable} if every $\J$-theory has free algebras (in a suitable enriched sense), and is \emph{strongly amenable} if every $\J$-pretheory has free algebras. As we outline in more detail below, we establish wide classes of examples of amenable and strongly amenable subcategories of arities, including all the subcategories of arities employed in the works of Lawvere \cite{Law:PhD}, Linton \cite{Lintonequational}, Dubuc \cite{Dubucsemantics}, Borceux-Day \cite{BorceuxDay}, Power \cite{PowerLawvere}, Nishizawa-Power \cite{NishizawaPower}, Lack-Rosick\'y \cite{LR}, Lucyshyn-Wright \cite{EAT}, and Bourke-Garner \cite{BourkeGarner}. 

Given an amenable subcategory of arities $j : \J \hookrightarrow \C$, we establish a structure--semantics adjunction between the opposite of the category $\Th_{\underJ}(\C)$ of $\J$-theories and the category $\J\Tract(\C)$ of \emph{$\J$-tractable} $\V$-categories over $\C$. This adjunction is idempotent, and restricts to a dual equivalence between $\Th_{\underJ}(\C)$ and the full subcategory $\J\Alg^!(\C) \hookrightarrow \J\Tract(\C)$ consisting of the \emph{strictly $\J$-algebraic} $\V$-categories over $\C$. We establish an intrinsic characterization of $\J$-algebraic $\V$-categories over $\C$ in terms of criteria that generalize the conclusions of Weber's \emph{nerve theorem} \cite[Theorem 4.10]{Webernerve} and so generalize the Bourke-Garner notion of \emph{$\J$-nervous $\V$-monad}, a term that we adopt also in the present setting. We also show that a $\V$-category over $\C$ is $\J$-algebraic iff it is monadic by way of a $\J$-nervous $\V$-monad.  Strikingly, our characterization theorem for $\J$-algebraic $\V$-categories over $\C$ does not directly mandate monadicity in any evident way but nonetheless entails it, so providing a novel \textit{$\J$-nervous monadicity theorem}. We obtain an equivalence $\Th_{\underJ}(\C) \simeq \Mnd_{\underJ}(\C)$ between the category of $\J$-theories and the category $\Mnd_{\underJ}(\C)$ of $\J$-nervous $\V$-monads on $\C$, and this equivalence commutes (up to isomorphism) with semantics in an appropriate sense.

When the subcategory of arities $j : \J \hookrightarrow \C$ is \emph{strongly} amenable, the above structure--semantics adjunction extends to (the opposite of) the category $\Preth_{\underJ}(\C)$ of $\J$-pretheories. We also obtain an idempotent adjunction between $\J$-pretheories and arbitrary $\V$-monads on $\C$, which generalizes the monad--pretheory adjunction established by Bourke and Garner \cite{BourkeGarner} in the locally presentable enriched setting.   
Under the additional assumptions that $\V$ is complete and cocomplete and that $\C$ is complete, we also show that $\Preth_{\underJ}(\C), \Th_{\underJ}(\C)$, and $\Mnd_{\underJ}(\C)$ are all cocomplete, with small colimits therein being \emph{algebraic}, i.e.~sent to limits in $\V\CAT/\C$ by the respective semantics functors. 

We now discuss the classes of examples of amenable and strongly amenable subcategories of arities that we shall establish. We show that every \emph{eleutheric} \cite{EAT, Pres} subcategory of arities $j : \J \hookrightarrow \C$ in an arbitrary $\V$-category $\C$ enriched in a closed category $\V$ with equalizers is amenable, and that (in this context) $\J$-nervous $\V$-monads coincide with the \emph{$\J$-ary} $\V$-monads of \cite{EAT, Pres}. We thus recover the monad--theory equivalence of the first author \cite{EAT}, and the several earlier equivalences that it generalizes, as well as the above structure--semantics adjunctions of Lawvere, Linton, Dubuc, and Borceux-Day. 

Under suitable (co)completeness assumptions on $\C$ and $\V$, we then show that if $j : \J \hookrightarrow \C$ is contained in some eleutheric and \emph{bounded} subcategory of arities \cite{Pres}, then $\J$ is strongly amenable. Most of our examples of eleutheric subcategories of arities also satisfy this further boundedness condition, and hence are strongly amenable. In particular, we remark that any small subcategory of arities in a locally presentable $\V$-category $\C$ enriched in a locally presentable closed category $\V$ is contained in some bounded and eleutheric subcategory of arities, and hence is strongly amenable, which allows us to recover results of Bourke and Garner \cite{BourkeGarner}, including their monad--pretheory adjunction and monad--theory equivalence. 

As another new class of examples, we show that if $j : \J \hookrightarrow \C$ is any small subcategory of arities in a \emph{$\V$-sketchable} $\V$-category $\C$ enriched in a \emph{locally bounded} closed category $\V$, then $\J$ is strongly amenable. Locally bounded closed categories include not only all locally presentable closed categories, but also all symmetric monoidal closed topological categories over $\Set$, including many convenient categories in topology and analysis.  $\V$-sketchable $\V$-categories are the $\V$-categories of structures in $\V$ describable by enriched limit theories. In particular, $\V$ itself is (trivially) $\V$-sketchable, and we deduce that every small subcategory of arities in a locally bounded closed category $\V$ is strongly amenable. Since every locally presentable closed category $\V$ is locally bounded and every locally presentable $\V$-category $\C$ is $\V$-sketchable, we thus obtain another way of recovering results of Bourke and Garner \cite{BourkeGarner}. 

We now briefly discuss some further related works that \emph{a priori} neither generalize nor specialize the present work. In \cite{Street_monads}, Street establishes a general 2-categorical structure--semantics adjunction for monads in a 2-category; however, it is neither clear nor known that the $\J$-nervous $\V$-monads considered in the present paper can be described as (precisely) the monads in a suitable 2-category. In \cite{formal_relative_monads}, Arkor and McDermott extend Street's work in \cite{Street_monads} by developing a 2-categorical treatment of \emph{relative monads}. In the PhD thesis \cite{Arkor_thesis}, Arkor defines a 2-categorical notion of \emph{theory} and uses results from \cite{formal_relative_monads} to establish an equivalence between theories and (relative) monads in a general 2-categorical setting. Arkor then shows that this equivalence specializes to obtain certain instances of the monad--theory equivalence established by the first author in \cite{EAT}, as well as the monad--theory equivalence established by Bourke and Garner in the locally presentable setting in \cite{BourkeGarner}. It is \emph{a priori} possible that some of our results herein could be placed in the 2-categorical context of Arkor's thesis, but this is not yet known. Furthermore, Arkor's thesis does not treat structure--semantics adjunctions (although this is mentioned therein as a direction for future work). Earlier, in the PhD thesis \cite{Avery_thesis}, Avery also defined a general notion of theory in a 2-categorical setting and established structure--semantics adjunctions for such theories; however, in view of the discussion in \cite[6.6, p. 119]{Avery_thesis}, the latter work does not demonstrate that its framework encompasses the enriched Lawvere theories of Nishizawa and Power \cite{NishizawaPower}, which are examples of the more general $\J$-theories considered in the present paper.  Also, Fujii \cite{Fujii_framework} defined a general framework for notions of algebraic theory formulated as monoids in a monoidal category, but it is neither clear nor known that the $\J$-theories considered in the present paper can be described in general as (precisely) the monoids in a monoidal category.

We now outline the paper. In \S\ref{summary} we provide a precise summary of our main results and examples, with references to later definitions and theorems, while \S\ref{background} pertains to notation and background material. In \S\ref{strsemsection} and \S\ref{mnd_th_section} we establish the structure--semantics adjunction and monad--theory equivalence for an amenable subcategory of arities, and we also establish some additional results for small and strongly amenable subcategories of arities. \S\ref{amenableexamples} is devoted to establishing several classes of examples of amenable and strongly amenable subcategories of arities, as previewed above.

In a subsequent paper \cite{Ax}, we shall show (under mild assumptions on $\C$ and $\V$) that small and strongly amenable subcategories of arities $j : \J \hookrightarrow \C$ admit many further useful results pertaining to \emph{presentations} of $\J$-nervous $\V$-monads, which will provide very convenient methods for constructing such monads. After adapting techniques from the authors' work \cite{EP} to the more general context of the present paper, we shall then show that the strong amenability of a small subcategory of arities $j : \J \hookrightarrow \C$ is actually \emph{equivalent} to $\J$ ``supporting presentations" in a certain natural sense. This result will provide further proof of concept for the axiomatics developed in the present paper.    

\bigskip

\noindent\textbf{Acknowledgment.} The authors thank the anonymous referee for various helpful suggestions and for pointing out an error in an earlier version of Proposition \ref{discrete_iso_prop}.

\section{Summary of main results and examples}
\label{summary}

For the reader's convenience, we begin with a precise summary our main results and examples. Most of the concepts involved have been mentioned in \S \ref{sec:intro}, and for convenience we include superscripts referring the reader to the definitions of terms defined later in the paper, as well as references to the theorems we now summarize.

\begin{theo}
\label{mainresults}
Let $\J \hookrightarrow \C$ be a (possibly large) subcategory of arities\textsuperscript{\textnormal{(\S \ref{strsemsection})}} in an arbitrary $\V$-category $\C$ enriched in a symmetric monoidal closed category $\V$ with equalizers.
\begin{enumerate}[leftmargin=*]
\item Suppose that $\J$ is amenable\textsuperscript{\eqref{tractablearities}}. Then we have an idempotent \textnormal{structure--semantics} adjunction between $\J$-theories\textsuperscript{\eqref{Jtheory}} and $\J$-tractable\textsuperscript{\eqref{Jtractable}} $\V$-categories over $\C$ \textnormal{(\ref{strsemadjunction}, \ref{idempotent_adjunction})}:
\[\begin{tikzcd}
	{\Th_{\underJ}(\C)^\op} &&& {} & {\J\Tract(\C)}.
	\arrow[""{name=0, anchor=center, inner sep=0}, "\Sem", shift left=3, from=1-1, to=1-5]
	\arrow[""{name=1, anchor=center, inner sep=0}, "\Str", shift left=3, from=1-5, to=1-1]
	\arrow["\dashv"{anchor=center, rotate=90}, draw=none, from=1, to=0]
\end{tikzcd} \]
From this idempotent adjunction, we obtain an equivalence
\[\begin{tikzcd}
	{\Th_{\underJ}(\C)} &&& {} & {\Mnd_{\underJ}(\C)}
	\arrow[""{name=0, anchor=center, inner sep=0}, "\sfm", shift left=3, from=1-1, to=1-5]
	\arrow[""{name=1, anchor=center, inner sep=0}, "\sft", shift left=3, from=1-5, to=1-1]
	\arrow["\sim"{anchor=center}, draw=none, from=1, to=0]
\end{tikzcd} \]
between $\J$-theories and $\J$-nervous\textsuperscript{\eqref{Jnervous}} $\V$-monads on $\C$, which commutes with semantics in an appropriate sense \eqref{monadtheoryequivalence}.

\item Suppose that $\V$ is complete and cocomplete, and that $\J$ is small and strongly amenable\textsuperscript{\eqref{tractablearities}}. Then we also have the following: 
\begin{enumerate}[leftmargin=*]
\item There is an idempotent \textnormal{structure--semantics} adjunction between $\J$-pretheories\textsuperscript{\eqref{Jtheory}} and $\J$-tractable $\V$-categories over $\C$ \eqref{strsemadjunction}:
\[\begin{tikzcd}
	{\Preth_{\underJ}(\C)^\op} &&& {} & {\J\Tract(\C)}.
	\arrow[""{name=0, anchor=center, inner sep=0}, "\Sem", shift left=3, from=1-1, to=1-5]
	\arrow[""{name=1, anchor=center, inner sep=0}, "\Str", shift left=3, from=1-5, to=1-1]
	\arrow["\dashv"{anchor=center, rotate=90}, draw=none, from=1, to=0]
\end{tikzcd} \]
\item There is an idempotent adjunction between $\J$-pretheories and $\V$-monads on $\C$ \eqref{strsemcorestricted}:
\[\begin{tikzcd}
	{\Preth_{\underJ}(\C)} &&& {} & {\Mnd(\C)}.
	\arrow[""{name=0, anchor=center, inner sep=0}, "\sfm", shift left=3, from=1-1, to=1-5]
	\arrow[""{name=1, anchor=center, inner sep=0}, "\sft", shift left=3, from=1-5, to=1-1]
	\arrow["\vdash"{anchor=center, rotate=90}, draw=none, from=1, to=0]
\end{tikzcd} \]
\item If $\C$ is complete, then each of $\Preth_{\underJ}(\C)$, $\Th_{\underJ}(\C)$, and $\Mnd_{\underJ}(\C)$ has small algebraic colimits \eqref{cocompletecor}.
\item Any small subcategory of arities contained in $\J$ is strongly amenable \eqref{pushoutthm}.
\end{enumerate}
\end{enumerate}
\end{theo}

\begin{theo}
\label{examplesthm}
We have the following classes of examples of amenable and strongly amenable subcategories of arities:
\begin{enumerate}[leftmargin=*]
\item Let $\V$ be a symmetric monoidal closed category with equalizers. Then every (possibly large) eleutheric\textsuperscript{\eqref{eleutheric}} subcategory of arities $\J \hookrightarrow \C$ in an arbitrary $\V$-category $\C$ is amenable \eqref{eleuthericamenable}.

\item Under the assumptions of \ref{boundedassumptions}, every subcategory of arities $\J \hookrightarrow \C$ that is contained in some bounded\textsuperscript{\eqref{boundedVfunctor}} and eleutheric subcategory of arities is strongly amenable \eqref{boundedeleuthericamenable}. In particular, given a locally bounded closed category\textsuperscript{\eqref{boundedexamples}} $\V,$ every small and eleutheric subcategory of arities $\J \hookrightarrow \C$ in a locally bounded $\V$-category\textsuperscript{\eqref{boundedexamples}} $\C$ is strongly amenable, as is every subcategory of arities contained in $\J$ \textnormal{(\ref{boundedeleuthericamenable}, \ref{boundedeleuthericexamples})}. 

\item Let $\V$ be a locally bounded closed category. Then every small subcategory of arities $\J \hookrightarrow \C$ in a $\V$-sketchable\textsuperscript{\eqref{locallyboundedassumptions}} $\V$-category $\C$ is strongly amenable \eqref{locbdtractable}. For example, this includes every small full sub-$\V$-category of $\V$ that contains the unit object \eqref{locbdVcase}.
\end{enumerate}
\end{theo}

\section{Notation and background}
\label{background}

We use the methods of enriched category theory throughout the paper; one can consult the texts \cite{Kelly, Dubucbook} or \cite[Chapter 6]{Borceux2} for details. We mainly use the notation of Kelly's text \cite{Kelly}. Throughout the paper, we let $\V = (\V_0, \tensor, I)$ be a symmetric monoidal closed category such that $\V_0$ is locally small and has equalizers; we typically just refer to $\V$ as a closed category. Since we do not assume that $\V_0$ is complete, we cannot in general form $\V$-functor $\V$-categories $[\A, \B]$ for $\V$-categories $\A$ and $\B$, even when $\A$ is small. We shall therefore make use of the method of \emph{universe extensions} described in \cite[\S 3.11, 3.12]{Kelly}, whereby $\V$ is embedded (via a symmetric strong monoidal functor) into a symmetric monoidal closed category $\V'$ that is $\scrU$-complete and $\scrU$-cocomplete with respect to some larger universe $\scrU$ of sets for which $\V_0$ is $\scrU$-small, where the given embedding preserves all limits and all $\scrU$-small colimits that exist in $\V$. All of the $\V$-categories that we consider are assumed to be $\scrU$-small. Given a $\V'$-category $\A$, we write $\A_0$ for its underlying ordinary category (which is locally $\scrU$-small). 

\section{The structure--semantics adjunction}
\label{strsemsection}

Throughout \S\ref{strsemsection}, we fix a \textbf{subcategory of arities} $j : \J \rightarrowtail \C$, i.e.~a fully faithful and dense $\V$-functor, in an arbitrary $\V$-category $\C$. For most purposes we can (and shall) assume that $j$ is the inclusion $\J \hookrightarrow \C$ of a full sub-$\V$-category. For example, the full sub-$\V$-category $\C \hookrightarrow \C$ is a (generally large) subcategory of arities. 

We have the \textbf{nerve $\V'$-functor} $N_j : \C \to \left[\J^\op, \V\right]$ defined by $N_j C = \C(j-, C)$ ($C \in \C$), which is fully faithful because $j : \J \hookrightarrow \C$ is dense. A \textbf{$j$-nerve} is then a presheaf $\J^\op \to \V$ in the essential image of $N_j$, and we write $j\Ner(\V) \hookrightarrow \left[\J^\op, \V\right]$ for the full sub-$\V'$-category of the $\V'$-category $\left[\J^\op, \V\right]$ consisting of the $j$-nerves. The fully faithful $N_j : \C \to \left[\J^\op, \V\right]$ thus corestricts to an equivalence $N_j : \C \xrightarrow{\sim} j\Ner(\V)$, so that $j\Ner(\V)$ is a $\V$-category.  

\begin{defn}
\label{Jtheory}
{
A \textbf{$\J$-pretheory}\footnote{Bourke and Garner \cite{BourkeGarner} use the dual notion of pretheory (under additional assumptions on $\J, \C, \V$), while our notion of pretheory accords with the notion of pretheory previously studied (in the unenriched context) by Linton \cite[Page 20]{Lintonoutline} under the name \emph{clone}, and later by Diers \cite{Diers} under the name \emph{theory}.} is a $\V$-category $\scrT$ equipped with an identity-on-objects $\V$-functor $\tau : \J^\op \to \scrT$. A \textbf{$\J$-theory} is a $\J$-pretheory $(\scrT, \tau)$ with the further property that each $\scrT(J, \tau-) : \J^\op \to \V$ ($J \in \ob\scrT = \ob\J$) is a $j$-nerve. A \textbf{morphism of $\J$-pretheories} from a $\J$-pretheory $(\scrT, \tau)$ to a $\J$-pretheory $(\mathscr{U}, \upsilon)$ is a $\V$-functor $H : \scrT \to \mathscr{U}$ satisfying $H \circ \tau = \upsilon$ (from which it follows that $H$ is identity-on-objects). We then have the ordinary category $\Preth_{\underJ}(\C)$ of $\J$-pretheories and their morphisms, as well as the full subcategory $\Th_{\underJ}(\C) \hookrightarrow \Preth_{\underJ}(\C)$ of $\J$-theories. Note that $\Preth_{\underJ}(\C)$ is just the full subcategory of the coslice category $\J^\op/\V\CAT$ consisting of the $\V$-functors that are identity-on-objects.  
}
\end{defn}

\begin{defn}
\label{concretealgebras}
{
Let $\scrT$ be a $\J$-pretheory. A \textbf{concrete $\scrT$-algebra}\footnote{This definition was also employed by Bourke and Garner \cite{BourkeGarner} (in their locally presentable setting), and goes back to Linton \cite{Lintonoutline} and Diers \cite{Diers}.} is an object $A$ of $\C$ equipped with a $\V$-functor $M : \scrT \to \V$ that extends the $j$-nerve $\C(j-, A) : \J^\op \to \V$ along $\tau : \J^\op \to \scrT$, i.e.~that satisfies $M \circ \tau = \C(j-, A)$. We shall denote a concrete $\scrT$-algebra by $(A, M)$ or simply by $A$. The $\V'$-category $\scrT\Alg^!$ of concrete $\scrT$-algebras is then defined by the following pullback diagram in $\V'\CAT$:
\begin{equation}\label{eqn:concrete_pb}
\begin{tikzcd}
	{\scrT\Alg^!} && {[\scrT, \V]} \\
	\\
	\C && {\left[\J^\op, \V\right]}.
	\arrow["{\sfM^\scrT}", from=1-1, to=1-3]
	\arrow["{[\tau, 1]}", from=1-3, to=3-3]
	\arrow["{U^\scrT}"', from=1-1, to=3-1]
	\arrow["{N_j}"', from=3-1, to=3-3]
\end{tikzcd}
\end{equation}
Note that $\sfM^\scrT : \scrT\Alg^! \to [\scrT, \V]$ is fully faithful, as a pullback of the fully faithful $N_j$. 
}
\end{defn} 

\noindent Although in general $\scrT\Alg^!$ is only a $\V'$-category, we shall principally be interested in cases where it is a $\V$-category; see Definition \ref{tractablearities} below.

\begin{defn}
\label{nonconcretealgebras}
{
Let $\scrT$ be a $\J$-pretheory. The $\V'$-category $\scrT\Alg$ of \textbf{(non-concrete) $\scrT$-algebras}\footnote{This definition was also employed by Bourke and Garner \cite{BourkeGarner} in their locally presentable setting.} is defined by the following pullback diagram in $\V'\CAT$:
\[\begin{tikzcd}
	{\scrT\Alg} && {[\scrT, \V]} \\
	\\
	j\Ner(\V) && {\left[\J^\op, \V\right]}.
	\arrow[hook, from=1-1, to=1-3]
	\arrow["{[\tau, 1]}", from=1-3, to=3-3]
	\arrow["{W^\scrT}"', from=1-1, to=3-1]
	\arrow[hook, from=3-1, to=3-3]
\end{tikzcd}\]
Thus, a (non-concrete) $\scrT$-algebra is a $\V$-functor $M : \scrT \to \V$ whose restriction along $\tau$ is a $j$-nerve. 
}
\end{defn}

\begin{para}
\label{system_arities_para}
{
In the setting where $\C = \V$, a \textbf{system of arities (in $\V$)} is a fully faithful strong symmetric monoidal $\V$-functor $j : \J \rightarrowtail \V$ (see \cite[Definition 3.1]{EAT}), so that $\J$ is in particular a symmetric monoidal $\V$-category. Every system of arities $j : \J \rightarrowtail \V$ is a dense $\V$-functor by \cite[Proposition 3.10]{EAT}, and thus is a subcategory of arities. In view of \cite[Proposition 3.8]{EAT} every system of arities is equivalent (in a suitable sense) to a full sub-$\V$-category $\J \hookrightarrow \V$ that contains the unit object $I$ and is closed under the monoidal product $\tensor$. Given a system of arities $\J \hookrightarrow \V$, the $\V$-category $\J$ has tensors by objects of $\J$, which are formed as in $\V$. It readily follows that a $\V$-functor $F : \J^\op \to \V$ is a $j$-nerve iff it preserves $\J$-cotensors (i.e.~cotensors by objects of $\J$). This observation then immediately entails the following result, which shows for a system of arities $j : \J \rightarrowtail \V$ that a $\J$-pretheory $\scrT$ is a $\J$-theory in the sense of Definition \ref{Jtheory} iff it is a $\J$-theory in the sense of \cite[Definition 4.1]{EAT}, and that a $\V$-functor $M : \scrT \to \V$ is a non-concrete algebra for a $\J$-theory $\scrT$ in the sense of Definition \ref{nonconcretealgebras} iff it is a $\scrT$-algebra in the sense of \cite[Definition 5.1]{EAT}.

\begin{prop}
\label{systemaritiesthm}
Let $j : \J \rightarrowtail \V$ be a system of arities, and let $\scrT$ be a $\J$-pretheory.
\begin{enumerate}[leftmargin=*]
\item A $\V$-functor $M : \scrT \to \V$ is a (non-concrete) $\scrT$-algebra iff $M \circ \tau : \J^\op \to \V$ preserves $\J$-cotensors. 
\item $\scrT$ is a $\J$-theory iff $\tau : \J^\op \to \scrT$ preserves $\J$-cotensors.
\item Suppose that $\scrT$ is a $\J$-theory. Then a $\V$-functor $M : \scrT \to \V$ is a (non-concrete) $\scrT$-algebra iff $M$ preserves $\J$-cotensors. \qed 
\end{enumerate}
\end{prop} 

} 
\end{para}

\begin{para}
\label{T_alg_repr_para}
{
Let $\scrT$ be a $\J$-pretheory. Given an object $(\A, G)$ of the slice category $\V'\CAT/\C$, it readily follows from the definition \eqref{concretealgebras} of $\scrT\Alg^!$ that morphisms $P : (\A, G) \to \left(\scrT\Alg^!, U^\scrT\right)$ in $\V'\CAT/\C$ correspond naturally and bijectively to $\V'$-functors $P_1 : \A \to [\scrT, \V]$ that make the following square commute:
\[\begin{tikzcd}
	\A && {[\scrT, \V]} \\
	\\
	\C && {\left[\J^\op, \V\right]}.
	\arrow["{P_1}", from=1-1, to=1-3]
	\arrow["{[\tau, 1]}", from=1-3, to=3-3]
	\arrow["G"', from=1-1, to=3-1]
	\arrow["{N_j}"', from=3-1, to=3-3]
\end{tikzcd}\]
By transposition, morphisms $P : (\A, G) \to \left(\scrT\Alg^!, U^\scrT\right)$ in $\V'\CAT/\C$ therefore correspond to $\V'$-functors $P_2 : \scrT \to [\A, \V]$ for which $P_2 \circ \tau = \C(j-, G?) : \J^\op \to [\A, \V]$.  
}
\end{para}

\begin{para}
\label{discrete_iso}
{
A \emph{discrete isofibration}  is a functor $G : \X \to \Y$ with the property that for any isomorphism $f : Y \xrightarrow{\sim} GX$ in $\Y$, there is a unique isomorphism $f' : X' \xrightarrow{\sim} X$ in $\X$ with $G(f') = f$. A faithful functor is a discrete isofibration iff it is \emph{uniquely transportable} in the sense of \cite[Definition 5.28]{AHS}. For example, every strictly monadic functor is a discrete isofibration (see e.g.~\cite[Proposition 20.12]{AHS}). We shall say that an enriched functor (i.e.~a $\V$- or $\V'$-functor) is a discrete isofibration if its underlying ordinary functor is a discrete isofibration. Given an identity-on-objects $\V$-functor $\tau : \A \to \B$ and a $\V'$-category $\C$, it is well known (and not difficult to prove) that the $\V'$-functor $[\tau, 1] : [\B, \C] \to [\A, \C]$ given by precomposition with $\tau$ is a discrete isofibration. In particular, if $\scrT$ is a $\J$-pretheory, then $\left[\tau, 1\right] : [\scrT, \V] \to \left[\J^\op, \V\right]$ is a discrete isofibration. Discrete isofibrations are also stable under pullback (see, e.g.,~the proof of \cite[Lemma 14]{BourkeGarner}), so that $U^\scrT : \scrT\Alg^! \to \C$ is a discrete isofibration for each $\J$-pretheory $\scrT$.

By \cite[Proposition 5.36]{AHS}, every faithful functor factors as an equivalence followed by a discrete isofibration. By taking a pseudo-inverse, it follows that for every faithful functor $G:\X \to \Y$ there is an equivalence $L:\X' \xrightarrow{\sim} \X$ such that $GL:\X' \to \Y$ is isomorphic to a discrete isofibration. We now extend this to the enriched setting by using the following basic transport-of-structure principles for $\V$-categories. Firstly, the 2-functor $(-)_0:\V\CAT \to \mathsf{CAT} = \Set\CAT$ is \textit{locally a discrete isofibration} in the sense that for all $\V$-categories $\X$ and $\Y$ the functor $(-)_0:\V\CAT(\X,\Y) \to \mathsf{CAT}(\X_0,\Y_0)$ is a discrete isofibration. Secondly, the functor $\ob:\V\CAT_0 \to \SET$ is a cloven fibration, where $\V\CAT_0$ is the underlying ordinary category of $\V\CAT$ and $\SET$ is the category of large sets: Explicitly, the $\ob$-cartesian morphisms are the fully faithful $\V$-functors, and for each $\V$-category $\X$ and each morphism $f:S \to \ob\X$ in $\SET$ we obtain an $\ob$-cartesian morphism $f^\X:f^*(\X) \to \X$ in $\V\CAT_0$ with $\ob f^\X = f$ as follows: The $\V$-category $f^*(\X)$ has $\ob f^*(\X) = S$ and $f^*(\X)(s,t) = \X(f(s),f(t))$ for all $s,t \in S$, with composition and identities as in $\X$, and the $\V$-functor $f^\X$ is given on objects by $f$ and on homs by identities. In particular, every $\V$-functor $F:\B \to \X$ factors as an identity-on-objects $\V$-functor $E:\B \to (\ob F)^*(\X)$ followed by a fully faithful $\V$-functor $M = (\ob F)^\X:(\ob F)^*(\X) \to \X$.
}
\end{para}

\begin{prop}
\label{discrete_iso_prop}
Let $G : \X \to \Y$ be a $\V$-functor whose underlying ordinary functor is faithful. Then there is a $\V$-category $\X'$ and an equivalence of $\V$-categories $L : \X' \xrightarrow{\sim} \X$ such that $G  L : \X' \to \Y$ is isomorphic to a discrete isofibration $G':\X' \to \Y$.
\end{prop}

\begin{proof}
By \ref{discrete_iso}, there is an equivalence of (ordinary) categories $K:\A \xrightarrow{\sim} \X_0$ and a discrete isofibration $H:\A \to \Y_0$ with $G_0 K \cong H$. Factoring $K$ as an identity-on-objects functor $E:\A \to (\ob K)^*(\X_0)$ followed by a fully faithful functor $M = (\ob K)^{\X_0}:(\ob K)^*(\X_0) \to \X_0$ with the notation of \ref{discrete_iso}, we find that $E$ is an isomorphism and $M$ is an equivalence. But the functor $(-)_0:\V\CAT_0 \to \mathsf{CAT}_0$ commutes with the cloven fibrations $\ob:\V\CAT_0 \to \SET$ and $\ob:\mathsf{CAT}_0 \to \SET$ and preserves the cleavages, so $(\ob K)^*(\X_0)$ underlies a $\V$-category $\X' := (\ob K)^*(\X)$, and $M$ underlies a fully faithful $\V$-functor $L := (\ob K)^\X:\X' \to \X$. But $L$ is also essentially surjective on objects since $M$ is so, and hence $L$ is an equivalence. On the other hand $H \cong G_0 K = G_0 ME = G_0 L_0E:\A \to \Y_0$ and $E:\A \to \X'_0$ is an isomorphism, so the $\V$-functor $GL:\X' \to \Y$ has $(GL)_0 = G_0L_0 \cong HE^{-1}:\X'_0 \to \Y_0$, noting that $HE^{-1}$ is a discrete isofibration since $H$ is so. But $(-)_0$ is locally a discrete isofibration (\ref{discrete_iso}), so there is a $\V$-functor $G':\X' \to \Y$ with $G'_0 = HE^{-1}$ and $G' \cong GL$.
\end{proof}

\begin{para}
\label{concreteequivalence}
{
Let $\scrT$ be a $\J$-pretheory. In view of Definitions \ref{concretealgebras} and \ref{nonconcretealgebras}, we obtain a comparison $\V'$-functor $\scrT\Alg^! \to \scrT\Alg$ as in the following commutative diagram in $\V'\CAT$, whose left square is thus (also) a pullback:
\[\begin{tikzcd}
	{\scrT\Alg^!} && \scrT\Alg && {[\scrT, \V]} \\
	\\
	\C && {j\Ner(\V)} && {\left[\J^\op, \V\right]}.
	\arrow[hook, from=1-3, to=1-5]
	\arrow["{W^\scrT}", from=1-3, to=3-3]
	\arrow["{[\tau, 1]}", from=1-5, to=3-5]
	\arrow[hook, from=3-3, to=3-5]
	\arrow[dashed, from=1-1, to=1-3]
	\arrow["{U^\scrT}", from=1-1, to=3-1]
	\arrow["{N_j}", from=3-1, to=3-3]
	\arrow["{\sfM^\scrT}", curve={height=-18pt}, from=1-1, to=1-5]
\end{tikzcd}\]
We now show that this comparison $\V'$-functor is an equivalence\footnote{Bourke and Garner established an analogous result \cite[Proposition 25]{BourkeGarner} in their locally presentable setting.}. It is fully faithful, as a pullback of the fully faithful $N_j : \C \xrightarrow{\sim} j\Ner(\V)$. It is also essentially surjective, because if $M : \scrT \to \V$ is a non-concrete $\scrT$-algebra, then $M \circ \tau : \J^\op \to \V$ is a $j$-nerve, and so there is some $A \in \ob\C$ with $M \circ \tau \cong \C(j-, A)$. Since $[\tau, 1]$ is a discrete isofibration \eqref{discrete_iso}, it then follows that $M$ is isomorphic to the image of a concrete $\scrT$-algebra under the comparison $\V'$-functor. So $\scrT\Alg^! \simeq \scrT\Alg$ in $\V'\CAT/\C$ when we equip $\scrT\Alg$ with the composite $\V'$-functor $\scrT\Alg \xrightarrow{W^\scrT} j\Ner(\V) \xrightarrow{\sim} \C$.
}
\end{para} 

\begin{para}
\label{J_rel_adj}
Let $G : \A \to \C$ and $F:\J \to \A$ be $\V$-functors. Applying Diers' notion of relative adjunction \cite{Diers} in the present setting of a $\V$-enriched subcategory of arities $j:\J \hookrightarrow \C$, we say that $F$ is a \textbf{$j$-relative left adjoint} for $G$ if $F$ is equipped with isomorphisms $\A(FJ, A) \cong \C(jJ, GA)$ $\V$-natural in $J \in \J, A \in \A$. Equivalently, a $j$-relative left adjoint for $G$ is a $\V$-functor $F : \J \to \A$ equipped with a $\V$-natural transformation $\eta : j \Rightarrow GF$ such that
\[ \A(FJ, A) \xrightarrow{G_{FJ, A}} \C(GFJ, GA) \xrightarrow{\C\left(\eta_J, 1\right)} \C(jJ, GA) \] is an isomorphism for all $J \in \ob\J$ and $A \in \ob\A$. In this situation, we say that \emph{$\eta : j \Rightarrow GF$ exhibits $F$ as a $j$-relative left adjoint for $G$}, or that $\eta$ is the \textbf{unit} of the $j$-relative adjunction.  We also say that $G$ is a \textbf{$j$-relative right adjoint} for $F$.

In particular, if $G$ has a left adjoint $F : \C \to \A$ with unit $\eta : 1 \Rightarrow GF$, then $\eta j : j \Rightarrow GFj$ exhibits $Fj : \J \to \A$ as a $j$-relative left adjoint for $G$, and $1:Fj \Rightarrow Fj$ exhibits $F$ as a left Kan extension of $Fj$ along $j$ (since $F$ preserves colimits and $1_\C = \Lan_j j$ by the density of $j$). Conversely, if $\eta : j \Rightarrow GF$ exhibits $F : \J \to \A$ as a $j$-relative left adjoint for $G$ and the left Kan extension $\Lan_j F : \C \to \A$ exists, with unit $\varphi : F \overset{\sim}{\Longrightarrow} \left(\Lan_j F\right)j$, then there is a unique $\V$-natural transformation $\overline{\eta} : 1 \Rightarrow G\Lan_j F$ that makes the triangle
\begin{equation}\label{eq:rel_adj_unit}
\begin{tikzcd}
	j \\
	GF && {G\left(\Lan_j F\right)j}
	\arrow["\eta"', Rightarrow, from=1-1, to=2-1]
	\arrow[""{name=0, anchor=center, inner sep=0}, "{\overline{\eta}j}", Rightarrow, from=1-1, to=2-3]
	\arrow["G\varphi"', Rightarrow, from=2-1, to=2-3]
	\arrow["{=}"{description}, draw=none, from=2-1, to=0]
\end{tikzcd}
\end{equation} 
commute, and moreover $\overline{\eta}$ exhibits $\Lan_jF$ as left adjoint to $G$. The $\V$-natural transformation $\overline{\eta}$ is obtained using the fact that $1 : j \Rightarrow  j$ presents $1_\C$ as $\Lan_j j$ (because $j$ is dense).   
\end{para}

\begin{para}
\label{Jtheoryfunctor}
Let $\scrT$ be a $\J$-pretheory, so that $\scrT$ is equipped with an identity-on-objects $\V$-functor $\tau:\J^\op \to \scrT$.  Note that 
$$
\begin{minipage}{5in}\textit{$\scrT$ is a $\J$-theory iff $\tau^\op:\J \to \scrT^\op$ is a $j$-relative left adjoint.}
\end{minipage}
$$
Now assuming that $\scrT$ is a $\J$-theory, we write 
$$S_\scrT \;:\; \scrT^\op \longrightarrow \C$$
to denote the $j$-relative right adjoint for $\tau^\op$.  Explicitly, this $\V$-functor $S_\scrT$ is characterized uniquely up to isomorphism by the statement that $\scrT(J, \tau-) \cong \C(j-, S_\scrT J)$, $\V$-naturally in $J \in \scrT^\op$. We also define a $\V$-functor
$$T_\scrT : \J \longrightarrow \C$$
as the composite $\J \xrightarrow{\tau^\op} \scrT^\op \xrightarrow{S_\scrT} \C$, which is, up to isomorphism, the unique $\V$-functor such that $\scrT(\tau J, \tau-) \cong \C(j-, T_\scrT J)$, $\V$-naturally in $J \in \J$. We write
$$u^{\scrT} \;:\; j \Longrightarrow S_{\scrT}\circ\tau^\op = T_\scrT$$
to denote the unit of the $j$-relative adjunction, and we call $u^{\scrT}$ the \textbf{unit of the \mbox{$\J$-theory} $\scrT$}.  Explicitly, the component $u^{\scrT}_J:J \to T_\scrT J$ at each $J \in \ob\J$ is the image of the identity morphism $1_J : J \to J$ under the bijection $\scrT_0(J, J) = \scrT_0(J, \tau J) \xrightarrow{\sim} \C_0(J, T_\scrT J)$. 

Writing $\y : \scrT^\op \to [\scrT, \V]$ for the Yoneda embedding, since $N_j \circ S_\scrT \cong [\tau, 1] \circ \y : \scrT^\op \to \left[\J^\op, \V\right]$ and $[\tau, 1]$ is a discrete isofibration \eqref{discrete_iso}, we can replace $\y : \scrT^\op \to [\scrT, \V]$ by an isomorphic $\V'$-functor $\y' : \scrT^\op \to [\scrT, \V]$ that satisfies $N_j \circ S_\scrT = [\tau, 1] \circ \y'$. We then obtain a fully faithful $\V'$-functor $\mathsf{Y}_\scrT : \scrT^\op \to \scrT\Alg^!$ by applying the universal property of $\scrT\Alg^!$ as a (strict) pullback:
\[\begin{tikzcd}
	{\scrT^\op} \\
	& {\scrT\Alg^!} && {[\scrT, \V]} \\
	\\
	& \C && {\left[\J^\op, \V\right]}.
	\arrow["{\sfM^\scrT}", from=2-2, to=2-4]
	\arrow["{[\tau, 1]}", from=2-4, to=4-4]
	\arrow["{U^\scrT}"', from=2-2, to=4-2]
	\arrow["{N_j}"', from=4-2, to=4-4]
	\arrow["{S_\scrT}"', curve={height=12pt}, from=1-1, to=4-2]
	\arrow["\y'", curve={height=-12pt}, from=1-1, to=2-4]
	\arrow["{\mathsf{Y}_\scrT}"{description}, dotted, from=1-1, to=2-2]
\end{tikzcd}\]
In other words, $\mathsf{Y}_\scrT : \scrT^\op \to \scrT\Alg^!$ is the morphism of $\V\CAT/\C$ corresponding to $\y' : \scrT^\op \to [\scrT, \V]$ via \ref{T_alg_repr_para}. From \cite[Proposition 5.16]{Kelly} we then have $N_{\mathsf{Y}_\scrT} \cong \sfM^\scrT : \scrT\Alg^! \to [\scrT, \V]$ (so that $\mathsf{Y}_\scrT$ is also dense, because $\sfM^\scrT$ is fully faithful). We write $$\phi_\scrT \;:\; \J \longrightarrow \scrT\Alg^!$$
for the composite $\V'$-functor $\J \xrightarrow{\tau^\op} \scrT^\op \xrightarrow{\mathsf{Y}_\scrT} \scrT\Alg^!$.  Note that $U^\scrT \circ \phi_\scrT = T_\scrT:\J \to \C$.
\end{para}

\noindent We now give one of the central definitions of the paper:

\begin{defn}
\label{tractablearities}
{
A $\J$-pretheory $\scrT$ is \textbf{admissible} if the $\V'$-category $\scrT\Alg^!$ of concrete $\scrT$-algebras is a $\V$-category\footnote{We shall see in Remark \ref{monadic_rmk} that this first requirement is redundant.} and the $\V$-functor $U^\scrT : \scrT\Alg^! \to \C$ has a left adjoint. The subcategory of arities $\J \hookrightarrow \C$ is \textbf{amenable} (resp.~\textbf{strongly amenable}) if every $\J$-theory (resp.~$\J$-pretheory) is admissible. We write $\Preth_{\underJ}^\sfa(\C)$ for the full subcategory of $\Preth_{\underJ}(\C)$ consisting of the admissible $\J$-pretheories.
}
\end{defn}

\begin{assumption}
\label{amenable_assn}
We shall assume for the remainder of \S\ref{strsemsection} that the subcategory of arities $\J \hookrightarrow \C$ is amenable.
\end{assumption}

\noindent A \textbf{$\V$-category over $\C$} is an object of the (strict) slice category $\V\CAT/\C$, i.e.~a $\V$-category $\A$ equipped with a $\V$-functor $G : \A \to \C$. We denote such a $\V$-category over $\C$ by $(\A, G)$ or even just $\A$. 

\begin{defn}
\label{Jtractable}
{
A $\V$-functor $G : \A \to \C$ with codomain $\C$ is \textbf{$\J$-tractable} if $\C$ admits the weighted limit $\{\C(J, G-), G\}$ for each $J \in \ob\J$. A $\V$-category $(\A, G)$ over $\C$ is \textbf{$\J$-tractable} if the associated $\V$-functor $G : \A \to \C$ is $\J$-tractable. We write $\J\Tract(\C)$ for the full subcategory of $\V\CAT/\C$ consisting of the $\J$-tractable $\V$-categories over $\C$. 
}
\end{defn}

\noindent We show in \ref{structure_facts_prop} that a $\V$-functor $G:\A \to \C$ is $\J$-tractable iff $G$ has a \textit{structure $\J$-theory} in the sense of \ref{new_structure}. The following fact is an immediate consequence of the relevant definitions: 

\begin{prop}
\label{Jtractablelem}
Let $G : \A \to \C$ be a $\J$-tractable $\V$-functor. Then the object of $\V$-natural transformations $[\A, \V]\left(\C(J, G-), \C(K, G-)\right)$ exists in $\V$ for all $J, K \in \ob\J$. \qed    
\end{prop}

\begin{rmk}
\label{Dubuctractable}
{
In \cite[\S 2]{Dubucsemantics}, Dubuc defines a $\V$-functor $G : \A \to \C$ to be \emph{strongly tractable} if the (pointwise) right Kan extension $\Ran_GG$ of $G$ along itself exists, which (by the dual of \cite[4.18]{Kelly}) is equivalent to $\C$ admitting the weighted limit $\{\C(C, G-), G\}$ for each $C \in \ob\C$. So $G$ is strongly tractable in Dubuc's sense iff $G$ is $\C$-tractable (for the subcategory of arities $\C \hookrightarrow \C$) in the sense of Definition \ref{Jtractable}.
}
\end{rmk}

\begin{lem}
\label{adjointtractable}
Let $G : \A \to \C$ be a $\V$-functor with a $j$-relative left adjoint $F : \J \to \A$. Then $G : \A \to \C$ is $\J$-tractable. In particular, if $G : \A \to \C$ has a left adjoint, then $G$ is $\J$-tractable.  
\end{lem}

\begin{proof}
For each $J \in \ob\J$ we have $\C(J, G-) \cong \A(FJ, -) : \A \to \V$, and hence $\C$ admits the weighted limit $\{\C(J, G-), G\}$ because $\{\C(J, G-), G\} \cong \{\A(FJ, -), G\} \cong GFJ$ by \cite[(3.10)]{Kelly}.   
\end{proof}

\begin{rmk}
\label{phi_rmk}
{
Let $\scrT$ be a $\J$-theory. Recall from \ref{Jtheoryfunctor} that we have the $\V$-functor $\phi_\scrT = \mathsf{Y}_\scrT \circ \tau^\op : \J \to \scrT\Alg^!$, which we now show is a $j$-relative left adjoint for the $\V$-functor $U^\scrT : \scrT\Alg^! \to \C$. Indeed, because $\sfM^\scrT : \scrT\Alg^! \to [\scrT, \V]$ is fully faithful and $\sfM^\scrT \phi_\scrT J \cong \scrT(\tau J, -)$ ($J \in \J$), we have the following isomorphisms, which are $\V$-natural in $J \in \J$ and $A = (A, M) \in \scrT\Alg^!$:
\[ \scrT\Alg^!\left(\phi_\scrT J, A\right) \cong [\scrT, \V]\left(\scrT(\tau J, -), M\right) \cong M\tau J = \C\left(jJ, U^\scrT A\right). \] The $\V$-natural transformation $\eta : j \Rightarrow U^\scrT \phi_\scrT$ that exhibits $\phi_\scrT$ as a $j$-relative left adjoint for $U^\scrT$ \eqref{J_rel_adj} is precisely the unit $u^{\scrT} : j \Rightarrow T_\scrT = U^\scrT\phi_\scrT$ of $\scrT$ \eqref{Jtheoryfunctor}. 
}
\end{rmk}

Given an ordinary category $\X$, we shall say that a full subcategory $\G \hookrightarrow \X$ is \emph{strongly generating} if the representable functors $\X(G, -) : \X \to \Set$ ($G \in \ob\G$) are jointly conservative (i.e.~jointly reflect isomorphisms); this is equivalent to the conservativity of the nerve functor $N_i : \X \to \left[\G^\op, \Set\right]$ determined by the inclusion $i : \G \hookrightarrow \X$.   

\begin{lem}
\label{str_gen_lem}
Let $\X$ be a complete (ordinary) category, and suppose that $\G \hookrightarrow \X$ is strongly generating. Then the inclusion $\G \hookrightarrow \X$ preserves all small limits that exist. 
\end{lem}

\begin{proof}
Writing $i : \G \hookrightarrow \X$ for the inclusion, the nerve functor $N_i : \X \to \left[\G^\op, \Set\right]$ is conservative and preserves limits, and hence reflects small limits (by, e.g.,~\cite[Proposition 2.9.7]{Borceux1}). Now let $D$ be a small diagram in $\G$ with limit $\limit D$ in $\G$. To show that $i : \G \hookrightarrow \X$ preserves this limit, it suffices to show that $N_i \circ i : \G \to \left[\G^\op, \Set\right]$ preserves this limit; but this is true because $N_i \circ i \cong \y : \G \to \left[\G^\op, \Set\right]$ and $\y$ preserves all limits that exist.      
\end{proof}

\noindent We shall employ the following result in proving Proposition \ref{cocompletecor} below.

\begin{prop}
\label{JTract_lims}
Suppose that $\C$ and $\V$ are complete. Then the inclusion $\J\Tract(\C) \hookrightarrow \V\CAT/\C$ preserves all small limits that exist.  
\end{prop}

\begin{proof}
Since $\V$ is complete, it is well known that $\V\CAT$ is complete, so that $\V\CAT/\C$ is complete. The full subcategory of $\V\CAT$ consisting of the $\V$-categories with exactly two objects is strongly generating, from which it readily follows that the full subcategory $\G \hookrightarrow \V\CAT/\C$ consisting of the $\V$-functors whose domains have exactly two objects is strongly generating. Since $\C$ is complete, every $\V$-functor $G : \A \to \C$ such that $\A$ has exactly two objects is $\J$-tractable, so that $\J\Tract(\C)$ contains $\G$ and is thus itself strongly generating. The result now follows immediately from Lemma \ref{str_gen_lem}. 
\end{proof}

\begin{defn}
\label{semfunctor}
{
We define the \textbf{semantics} functor $\Sem : \Preth_{\underJ}^\sfa(\C)^\op \to \J\Tract(\C)$ as follows. For an admissible $\J$-pretheory $\scrT$, the $\V$-functor $U^\scrT : \scrT\Alg^! \to \C$ is $\J$-tractable by Lemma \ref{adjointtractable}, and we define $\Sem\scrT$ to be $\left(\scrT\Alg^!, U^\scrT\right)$.

Now let $H : (\scrT, \tau) \to (\scrU, \upsilon)$ be a morphism of admissible $\J$-pretheories. We define $\Sem H : \scrU\Alg^! \to \scrT\Alg^!$ in $\V\CAT/\C$ by using the universal property of $\scrT\Alg^!$, as in the following diagram in $\V'\CAT$:
\[\begin{tikzcd}
	{\scrU\Alg^!} && {[\scrU, \V]} \\
	& {\scrT\Alg^!} && {[\scrT, \V]} \\
	\\
	& \C && {\left[\J^\op, \V\right]}.
	\arrow["{\sfM^\scrT}", from=2-2, to=2-4]
	\arrow["{[\tau, 1]}", from=2-4, to=4-4]
	\arrow["{U^\scrT}"', from=2-2, to=4-2]
	\arrow["{N_j}"', from=4-2, to=4-4]
	\arrow["{\Sem H}"{description}, dashed, from=1-1, to=2-2]
	\arrow["{U^\scrU}"', curve={height=18pt}, from=1-1, to=4-2]
	\arrow["{\sfM^\scrU}", from=1-1, to=1-3]
	\arrow["{[H, 1]}", from=1-3, to=2-4]
	\arrow["{[\upsilon, 1]}"{description}, from=1-3, to=4-4]
\end{tikzcd}\]
It is readily verified that $\Sem$ is functorial.  
}
\end{defn}

\begin{defn}
\label{new_structure}
Let $G : \A \to \C$ be a $\V$-functor, so that $(\A, G)$ is a $\V$-category over $\C$. A \textbf{structure $\J$-theory for $G$} is a $\J$-theory $\left(\Str G, \tau_G\right)$ equipped with a fully faithful $\V'$-functor $m_G : \Str G \rightarrowtail [\A, \V]$ such that $m_G \circ \tau_G = \C(j-, G?) : \J^\op \to [\A, \V]$, so that the following triangle commutes:
\begin{equation}\label{eq:str}
\begin{tikzcd}
	{\J^\op} && {[\A, \V]}. \\
	& {\Str G}
	\arrow["{\C(j-, G?)}", from=1-1, to=1-3]
	\arrow["{\tau_G}"', from=1-1, to=2-2]
	\arrow["{m_G}"', from=2-2, to=1-3]
\end{tikzcd}
\end{equation}
We may also write $\Str(\A, G)$ for $\Str G$. 
\end{defn}

\begin{rmk}\label{rem:jstr_gtract}
Given an arbitrary $\V$-functor $G:\A \to \C$, we now address the question of existence of a structure $\J$-theory for $G$.  We can always factor $\C(j-,G?):\J^\op \to [\A,\V]$ as an identity-on-objects $\V'$-functor $\tau_G$ followed by a fully faithful $\V'$-functor $m_G$ as in \eqref{eq:str}, and such a factorization is unique up to an isomorphism that commutes with the $\V'$-functors $\tau_G$, $m_G$.  Such a factorization $\J^\op \xrightarrow{\tau_G} \Str G \xrightarrow{m_G} [\A,\V]$ is a structure $\J$-theory for $G$ if and only if the pair $(\Str G,\tau_G)$ is a $\J$-theory, noting that the latter condition requires for each object $J$ of $\J$ the existence of an object $C$ of $\C$ such that $\C(j-,C) \cong \Str G(J,\tau_G-):\J^\op \to \V$ (with $\Str G$ then automatically a $\V$-category).  But in any case
$$\Str G(J,\tau_G K) \cong [\A,\V](\C(J,G-),\C(jK,G-))\;,$$
$\V'$-naturally in $K \in \J$, so an object $C$ with the needed property is equivalently a limit $\{\C(J,G-),G\}$ in $\C$.  Thus we obtain the following result:
\end{rmk}

\begin{prop}
\label{structure_facts_prop}
Let $G : \A \to \C$ be a $\V$-functor. Then $G$ admits a structure $\J$-theory iff $G$ is $\J$-tractable.  A structure $\J$-theory $\Str G$ for $G$ is unique up to isomorphism in $\Th_\J(\C)$ if it exists, in which case any isomorphic $\J$-pretheory is also a structure $\J$-theory for $G$. 
\end{prop}

\begin{defn}
\label{unique_structure_defn}
Let $G : \A \to \C$ be a $\J$-tractable $\V$-functor.  Then, by Proposition \ref{structure_facts_prop}, a structure $\J$-theory $\Str G$ for $G$ exists and is unique up to isomorphism of $\J$-theories, so we call $\Str G$ \textbf{the structure $\J$-theory for $G$}, or \textbf{the $\J$-structure of $G$}. 
\end{defn} 

\begin{defn}
\label{unitmorphism}
{
Let $G : \A \to \C$ be a $\J$-tractable $\V$-functor, so that $(\A, G)$ is a $\J$-tractable $\V$-category over $\C$. The fully faithful $\V'$-functor $m_G : \Str G \rightarrowtail [\A, \V]$ associated to $\Str G$ satisfies $m_G \circ \tau_G = \C(j-, G?)$, and therefore corresponds by \ref{T_alg_repr_para} to a morphism \[ \sfE_{(\A, G)} : (\A, G) \to \left(\Str G\Alg^!, U^{\Str G}\right) \] in $\V\CAT/\C$. Explicitly, $\sfE_{(\A, G)}$ is the induced $\V$-functor in the following diagram
\[\begin{tikzcd}
	\A \\
	& {\Str G\Alg^!} && {[\Str G, \V]} \\
	\\
	& \C && {\left[\J^\op, \V\right]}
	\arrow["{\sfM^{\Str G}}", from=2-2, to=2-4]
	\arrow["{\left[\tau_G, 1\right]}", from=2-4, to=4-4]
	\arrow["{U^{\Str G}}"', from=2-2, to=4-2]
	\arrow["{N_j}"', from=4-2, to=4-4]
	\arrow["{\sfE_{(\A, G)}}"{description}, dashed, from=1-1, to=2-2]
	\arrow["G"', curve={height=18pt}, from=1-1, to=4-2]
	\arrow["\widetilde{m}_G", curve={height=-18pt}, from=1-1, to=2-4]
\end{tikzcd}\]      
where $\widetilde{m}_G$ is the transpose of $m_G$.
}
\end{defn}

We now establish a structure--semantics adjunction for amenable subcategories of arities.

\begin{theo}
\label{strsemadjunction}
Let $\J \hookrightarrow \C$ be an amenable subcategory of arities. Then the semantics functor $\Sem : \Preth_{\underJ}^\sfa(\C)^\op \to \J\Tract(\C)$ has a left adjoint \[ \Str : \J\Tract(\C) \to \Preth_{\underJ}^\sfa(\C)^\op \] that sends a $\J$-tractable $\V$-category $(\A, G)$ over $\C$ to its $\J$-structure $\Str (\A, G) = \Str G$ \eqref{unique_structure_defn}. The unit of this adjunction at $(\A, G)$ is
\[ \sfE_{(\A, G)} : (\A, G) \to \left(\Str G\Alg^!, U^{\Str G}\right) = \Sem\Str (\A, G). \] 
\end{theo}

\begin{proof}
The $\J$-structure $\Str G$ of $G$ is a $\J$-theory and so, by the amenability of $\J$, is an admissible $\J$-pretheory. We now show that $\sfE_{(\A, G)}$ is a $\Sem$-universal arrow for $(\A, G)$. So let $P : (\A, G) \to \Sem\scrT = \left(\scrT\Alg^!, U^\scrT\right)$ be a morphism in $\V\CAT/\C$ for an admissible $\J$-pretheory $\scrT$, and let us show that there is a unique morphism of $\J$-pretheories $P^\sharp : \scrT \to \Str G$ such that $\Sem P^\sharp \circ \sfE_{(\A, G)} = P$, all as in the following diagram:
\[\begin{tikzcd}
	{(\A, G)} && {\left(\Str G\Alg^!, U^{\Str G}\right)} & {\Str G} \\
	\\
	&& {\left(\scrT\Alg^!, U^\scrT\right)} & \scrT
	\arrow["{\sfE_{(\A, G)}}", from=1-1, to=1-3]
	\arrow["P"', from=1-1, to=3-3]
	\arrow["{\Sem P^\sharp}", from=1-3, to=3-3]
	\arrow["{P^\sharp}"{description}, dashed, from=3-4, to=1-4]
\end{tikzcd}\]
By \ref{T_alg_repr_para}, the morphism $P$ is equivalently given by a $\V'$-functor $P_2 : \scrT \to [\A, \V]$ such that $P_2 \circ \tau = \C(j-, G?) : \J^\op \to [\A, \V]$, so that the outer square of the following diagram in $\V'\CAT$ commutes (by the definition of $\Str G$):
\begin{equation}\label{str_sem_square}
\begin{tikzcd}
	{\J^\op} && \scrT \\
	\\
	{\Str G} && {[\A, \V]}.
	\arrow["\tau", from=1-1, to=1-3]
	\arrow["{\tau_G}"', from=1-1, to=3-1]
	\arrow["{P_2}", from=1-3, to=3-3]
	\arrow["{m_G}"', from=3-1, to=3-3]
	\arrow["{P^\sharp}"{description}, dashed, from=1-3, to=3-1]
\end{tikzcd}
\end{equation}
By the orthogonality of identity-on-objects $\V'$-functors to fully faithful $\V'$-functors, there is a unique $\V'$-functor (and hence $\V$-functor) $P^\sharp : \scrT \to \Str G$ that makes the diagram commute, so that $P^\sharp$ is a morphism of $\J$-pretheories. 
The commutativity of the lower triangle in \eqref{str_sem_square} is equivalent to the commutativity of the following square, which entails that $\Sem P^\sharp \circ \sfE_{(\A, G)} = P$, since the left-hand side of the latter equation corresponds via \ref{T_alg_repr_para} to the lower-left composite in this square.
\[\begin{tikzcd}
	\A && {\scrT\Alg^!} \\
	\\
	{\left[\Str G, \V\right]} && {[\scrT, \V]}
	\arrow["P", from=1-1, to=1-3]
	\arrow["{\widetilde{m}_G}"', from=1-1, to=3-1]
	\arrow["{\left[P^\sharp, 1\right]}"', from=3-1, to=3-3]
	\arrow["{\sfM^\scrT}", from=1-3, to=3-3]
\end{tikzcd}\]
The required uniqueness of $P^\sharp$ then follows from the uniqueness of the diagonal filler in \eqref{str_sem_square}.       
\end{proof}

\begin{prop}
\label{idempotent_adjunction}\mbox{}
\begin{enumerate}[leftmargin=*]
\item The structure--semantics adjunction \eqref{strsemadjunction} is idempotent.
\item Let $\scrT$ be an admissible $\J$-pretheory. Then $\scrT$ is a $\J$-theory iff the counit morphism $\varepsilon_\scrT : \scrT \to \Str\Sem\scrT$ is invertible.
\item The restriction of $\Sem$ to $\Th_{\underJ}(\C)^\op$ is a fully faithful functor $\Sem : \Th_{\underJ}(\C)^\op \to \J\Tract(\C)$ with a left adjoint $\Str : \J\Tract(\C) \to \Th_{\underJ}(\C)^\op$.
\end{enumerate}
\end{prop}

\begin{proof}
We first prove (2). The morphism $\varepsilon_\scrT : \scrT \to \Str\Sem\scrT$ is defined by taking $H$ to be the identity morphism of $\Sem\scrT$ in the proof of Theorem \ref{strsemadjunction}, so that $\varepsilon_\scrT : \scrT \to \Str\Sem\scrT$ is the unique $\V$-functor making the following diagram in $\V'\CAT$ commute:
\[\begin{tikzcd}
	{\J^\op} && \scrT \\
	\\
	{\Str \Sem\scrT} && {\left[\scrT\Alg^!, \V\right]},
	\arrow["\tau", from=1-1, to=1-3]
	\arrow["{\tau_{U^\scrT}}"', from=1-1, to=3-1]
	\arrow["{\widetilde{\sfM^\scrT}}", from=1-3, to=3-3]
	\arrow["{m_{U^\scrT}}"', from=3-1, to=3-3]
	\arrow["{\varepsilon_\scrT}"{description}, dashed, from=1-3, to=3-1]
\end{tikzcd}\]
where $\widetilde{\sfM^\scrT}$ is the transpose of the $\V'$-functor $\sfM^\scrT : \scrT\Alg^! \to [\scrT, \V]$.  

If $\varepsilon_\scrT$ is invertible, then $\scrT$ is certainly a $\J$-theory, because $\Str\Sem\scrT$ is a $\J$-theory. Suppose conversely that $\scrT$ is a $\J$-theory. Since $\varepsilon_\scrT$ is identity-on-objects, it suffices to show that $\varepsilon_\scrT$ is fully faithful. And since $m_{U^\scrT}$ is fully faithful, it then suffices to show that $\widetilde{\sfM^\scrT} : \scrT \to \left[\scrT\Alg^!, \V\right]$ is fully faithful. By \ref{Jtheoryfunctor}, the fully faithful $\V$-functor $\mathsf{Y}_\scrT : \scrT^\op \to \scrT\Alg^!$ satisfies $N_{\mathsf{Y}_\scrT} \cong \sfM^\scrT : \scrT\Alg^! \to [\scrT, \V]$, from which it follows that $\widetilde{\sfM^\scrT} : \scrT \to \left[\scrT\Alg^!, \V\right]$ is isomorphic to the fully faithful composite $\V'$-functor $\scrT \xrightarrow{\mathsf{Y}_\scrT^\op} \left(\scrT\Alg^!\right)^\op \xrightarrow{\y} \left[\scrT\Alg^!, \V\right]$.

To prove (1), it suffices to show that the whiskered counit $\varepsilon\Str : \Str\Sem\Str \Rightarrow \Str$ is a natural isomorphism; but if $(\A, G)$ is a $\J$-tractable $\V$-category over $\C$, then $\Str G$ is a $\J$-theory, so that $\varepsilon_{\Str G}$ is an isomorphism by (2). Lastly, (3) follows from (1) and (2) together with general facts about idempotent adjunctions. 
\end{proof}

\begin{defn}
\label{J_alg_Vcategory}
{
Let $G : \A \to \C$ be a $\V$-functor, so that $(\A, G)$ is a $\V$-category over $\C$. We say that $G$, or $(\A, G)$, is \textbf{strictly $\J$-algebraic} if there is an admissible $\J$-pretheory $\scrT$ satisfying $\A \cong \scrT\Alg^!$ in $\V\CAT/\C$. We say that $G$, or $(\A, G)$, is \textbf{$\J$-algebraic} if there is an admissible $\J$-pretheory $\scrT$ with an equivalence $\A \simeq \scrT\Alg^!$ in the pseudo-slice 2-category $\V\CAT//\C$ \cite[2.5]{EAT}. Every $\J$-algebraic $\V$-functor is $\J$-tractable by Lemma \ref{adjointtractable}, so the strictly $\J$-algebraic $\V$-categories over $\C$ constitute a full subcategory of $\J\Tract(\C)$ that we denote by $\J\Alg^!(\C)$.
} 
\end{defn}

The full subcategory $\J\Alg^!(\C) \hookrightarrow \J\Tract(\C)$ is the essential image of $\Sem$, so by \ref{idempotent_adjunction} we obtain the following:

\begin{prop}
\label{unitinvertible}
Let $G:\A \to \C$ be a $\V$-functor, so that $(\A,G)$ is a $\V$-category over $\C$.  The following are equivalent: (1) $G$ is strictly $\J$-algebraic; (2) there is a $\J$-theory $\scrT$ such that $\A \cong \scrT\Alg^!$ in $\V\CAT/\C$; (3) $G$ is $\J$-tractable and the unit morphism $\sfE_{(\A, G)} : (\A, G) \to \Sem\Str G$ of the structure--semantics adjunction \eqref{strsemadjunction} is invertible. \qed
\end{prop} 

\begin{cor}\label{thm:basic_charn_jalgcats}
A $\V$-functor $G:\A \to \C$ is $\J$-algebraic iff there is a $\J$-theory $\scrT$ with an equivalence $\A \simeq \scrT\Alg^!$ in $\V\CAT//\C$. Every $\J$-algebraic $\V$-functor is $\J$-tractable. \qed
\end{cor}

The following theorem now follows immediately from Propositions \ref{idempotent_adjunction} and \ref{unitinvertible}: 

\begin{theo}
\label{theory_Jalg_equiv}
Let $\J \hookrightarrow \C$ be an amenable subcategory of arities. Then the structure--semantics adjunction \eqref{strsemadjunction} restricts to an equivalence
\[\begin{tikzcd}
	{\Th_{\underJ}(\C)^\op} &&& {} & {\J\Alg^!(\C)}
	\arrow[""{name=0, anchor=center, inner sep=0}, "\Sem", shift left=3, from=1-1, to=1-5]
	\arrow[""{name=1, anchor=center, inner sep=0}, "\Str", shift left=3, from=1-5, to=1-1]
	\arrow["\sim"{anchor=center}, draw=none, from=1, to=0]
\end{tikzcd} \]
between $\J$-theories and strictly $\J$-algebraic $\V$-categories over $\C$. Furthermore, we have that $\J\Alg^!(\C) \hookrightarrow \J\Tract(\C)$ is reflective with reflector $\Sem \Str \, : \,\J\Tract(\C) \to \J\Alg^!(\C)$, and $\Th_{\underJ}(\C) \hookrightarrow \Preth_{\underJ}^\sfa(\C)$ is reflective with reflector $\Str\Sem: \Preth_{\underJ}^\sfa(\C) \to \Th_{\underJ}(\C)$. \qed   
\end{theo}

\begin{lem}
\label{modification_lem}
Let $\scrT$ be a $\J$-pretheory, let $G : \A \to \C$ be a $\V$-functor, let $m : \scrT \rightarrowtail [\A, \V]$ be a fully faithful $\V'$-functor, and let $\alpha$ be an invertible 2-cell in $\V'\CAT$ as in the leftmost diagram below. Then there is a unique pair $(m', \alpha')$ consisting of a fully faithful $\V'$-functor $m'$ and an invertible 2-cell $\alpha'$ as in the rightmost diagram below, such that the lower triangle in that diagram commutes and the resulting pasted 2-cell equals $\alpha$:
$$
\xymatrix{
\scrT \ar[rr]^m \ar@{}[rr]^(.2){}="s3" & & [\A,\V] & & \scrT \ar@/^4ex/[rr]^{m}="s1" \ar[rr]_{m'}="t1" & & [\A,\V]\\
\J^\op \ar[u]^\tau \ar[urr]_{\C(j-,G?)}="t3" & & & & \J^\op \ar[u]^\tau \ar[urr]_{\;\C(j-,G?)}
\ar@{}"s1";"t1"|(.35){}="s2"|(.65){}="t2"
\ar@{=>}"s2";"t2"^{\alpha'}
\ar@{}"s3";"t3"|(.3){}="s4"|(.6){}="t4"
\ar@{=>}"s4";"t4"^{\alpha}
}
$$
Consequently, if $G$ is $\J$-tractable, then $\scrT$ is a structure $\J$-theory for $G$ when equipped with $m' : \scrT \rightarrowtail [\A, \V]$.
\end{lem}

\begin{proof}
Writing $\ob\J$ for the discrete $\V$-category on the objects of $\J$ and similarly for $\ob\scrT$, the forgetful $\V'$-functors $\left[\J^\op, [\A, \V]\right] \to \left[\ob\J, [\A, \V]\right]$ and $\left[\scrT, [\A, \V]\right] \to \left[\ob\scrT, [\A, \V]\right]$ are discrete isofibrations by \ref{discrete_iso}, since they are given by precomposition with the canonical identity-on-objects $\V$-functors $\ob\J \to \J^\op$ and $\ob\scrT \to \scrT$. Since $\tau$ is identity-on-objects, the needed pair $(m', \alpha')$ can be obtained by using the second of these two discrete isofibrations, while the commutativity of the lower righthand triangle above may be proved using the first of these discrete isofibrations. We then deduce the final assertion from \ref{rem:jstr_gtract} and \ref{structure_facts_prop}.
\end{proof}

If a $\V$-functor $G:\A \to \C$ has a $j$-relative left adjoint $F$ \eqref{J_rel_adj}, then $G$ is $\J$-tractable (by \ref{adjointtractable}) and so has a structure $\J$-theory, by \ref{structure_facts_prop}. In this special case, the structure $\J$-theory $\Str G$ admits the following convenient description in terms of $F$, which we use in proving a characterization theorem for $\J$-algebraic $\V$-functors (Theorem \ref{J_alg_char_thm}).  

\begin{prop}
\label{Str_prop}
Let $G : \A \to \C$ be a $\V$-functor with a $j$-relative left adjoint $F : \J \to \A$. Let
\begin{equation}\label{eq:A_F}
\begin{tikzcd}
	{\J} && {\A} \\
	& {\A_F}
	\arrow["{F}", from=1-1, to=1-3]
	\arrow["{i_F}"', from=1-1, to=2-2]
	\arrow["{j_F}"', from=2-2, to=1-3]
\end{tikzcd}
\end{equation}
be the factorization of $F$ as an identity-on-objects $\V$-functor $i_F$ followed by a fully faithful $\V$-functor $j_F$. Then there is a fully faithful $\V'$-functor $m_G : \A_F^\op \to [\A, \V]$ that is isomorphic to the composite $\A_F^\op \xrightarrow{j_F^\op} \A^\op \xrightarrow{\y} [\A, \V]$ and makes the $\J$-pretheory $\left(\A_F^\op, i_F^\op\right)$ a structure $\J$-theory for $G$ \eqref{new_structure}. Thus we may take $\left(\Str G, \tau_G\right) = \left(\A_F^\op, i_F^\op\right)$.

Writing $\eta : j \Rightarrow GF$ for the $\V$-natural transformation that exhibits $F$ as a $j$-relative left adjoint for $G$ \eqref{J_rel_adj}, we can take $T_{\Str G} = GF:\J \to \C$, and we can take the unit $u^{\Str G} : j \Rightarrow T_{\Str G}$ \eqref{Jtheoryfunctor} of $\Str G$ to be $\eta$. With these choices, $\eta : j \Rightarrow GF = U^{\Str G}\sfE_{(\A, G)}F$ exhibits $\sfE_{(\A, G)}F$ as a $j$-relative left adjoint for $U^{\Str G}$.    
\end{prop}

\begin{proof}
We have isomorphisms
\begin{equation}\label{eq:jrel_adj_isos}\alpha_{JA}\;:\;\A(FJ,A) \overset{\sim}{\longrightarrow} \C(jJ,GA)\end{equation}
that are $\V$-natural in $J \in \J$ and $A \in \A$ and so constitute an isomorphism
$$\alpha\;:\;\y \circ F^\op \overset{\sim}{\Longrightarrow} \C(j-,G?)\;:\;\J^\op \longrightarrow [\A,\V],$$ where $\y:\A^\op \to [\A,\V]$ is the Yoneda embedding.  The $\V'$-functor
$$\ell := \y \circ j_F^\op\:\;\A_F^\op \longrightarrow [\A,\V]$$
is fully faithful, and its composite with $\tau_G = i_F^\op:\J^\op \to \Str G = \A_F^\op$ is $\ell \circ \tau_G = \y \circ j_F^\op \circ i_F^\op = \y \circ F^\op:\J^\op \to [\A,\V]$, so $\alpha$ is an isomorphism of the form
$$\alpha\;:\;\ell \circ \tau_G \overset{\sim}{\Longrightarrow} \C(j-,G?)\;:\;\J^\op \longrightarrow [\A,\V]\;.$$
Hence, by Lemma \ref{modification_lem} we obtain a fully faithful $\V'$-functor $m_G:\Str G \to [\A,\V]$ that makes $(\Str G,\tau_G) := (\A_F^\op,i_F^\op)$ a structure $\J$-theory for $G$, and we obtain also an isomorphism $\alpha':\ell \overset{\sim}{\Rightarrow} m_G$ such that $\alpha' \circ \tau_G = \alpha$. In particular, $m_G \circ \tau_G = \C(j-,G?):\J^\op \to [\A,\V]$.

Since $\Str G(\tau_G J,\tau_G K) = \A_F(i_F K,i_F J) = \A(FK,FJ)$ $(J,K \in \J)$, the isomorphisms
$$\alpha_{-,FJ}\;:\;\Str G(\tau_G J,\tau_G-) = \A(F-,FJ) \overset{\sim}{\longrightarrow} \C(j-,GFJ)\;\;\;\;(J \in \J)$$
witness that $\Str G$ is a $\J$-theory with $T_{\Str G} = GF$ and with unit $u^{\Str G} = \eta$. 

Note that $\alpha':\ell \overset{\sim}{\Rightarrow} m_G$ is literally the same family of isomorphisms $\alpha_{JA}$ that appear in \eqref{eq:jrel_adj_isos}, except now seen as a family 
$$\alpha'_{JA}\;:\;(\ell J)A = \A(j_F J,A) \overset{\sim}{\longrightarrow} (m_G J)A$$
$\V$-natural in $J \in \Str G$ and $A \in \A$.  Writing $\widetilde{\ell},\widetilde{m}_G:\A \to [\Str G,\V]$ to denote the transposes of $\ell$ and $m_G$, this same $\V$-natural family provides an isomorphism
$$\widetilde{\alpha}\;:\;\widetilde{\ell} \overset{\sim}{\Longrightarrow} \widetilde{m}_G\;.$$
To prove the final assertion, let us write $\sfE = \sfE_{(\A,G)}:\A \to \Str G\Alg^!$ and $\sfM = \sfM^{\Str G}:\Str G\Alg^! \to [\Str G,\V]$, and recall from \ref{unitmorphism} that $\sfM \sfE = \widetilde{m}_G$.  Hence $\widetilde{\alpha}:\widetilde{\ell} \overset{\sim}{\Rightarrow} \sfM\sfE$ and in particular $\widetilde{\alpha} F: \widetilde{\ell} F \overset{\sim}{\Rightarrow} \sfM \sfE F:\J \to [\Str G,\V]$.  But
$$\widetilde{\ell}FJ = (\ell-)FJ = \A(j_F-,FJ) = \A(j_F-,j_Fi_F J) = \Str G(\tau_G J,-)\;\;\;\;\;\;\;(J \in \J),$$
so $\widetilde{\alpha} F$ is a family of isomorphisms
\begin{equation}\label{eq:repr}\widetilde{\alpha}_{FJ} = \alpha'_{-,FJ}\;:\;\Str G(\tau_G J,-) \overset{\sim}{\Longrightarrow} (m_G -)FJ = \sfM\sfE FJ\end{equation}
$\V$-natural in $J \in \J$, each of which is a representation whose unit $I \to (m_G\tau_G J)FJ$ corresponds to the identity morphism on $\tau_G J$ in $\Str G$.  But the component 
$$\alpha'_{\tau_G J,FJ}\;:\;\Str G(\tau_G J,\tau_G J) \overset{\sim}{\longrightarrow} (m_G \tau_G J)FJ$$
of \eqref{eq:repr} that is obtained by evaluating at the object $\tau_G J$ is precisely
$$\alpha_{J,FJ}\;:\;\A(FJ,FJ) \overset{\sim}{\longrightarrow} \C(J,GFJ),$$
so the unit of the representation \eqref{eq:repr} is precisely the morphism $\eta_J:J \to GFJ$ in $\C$.

Using these observations, together with the fact that $\sfM:\Str G\Alg^! \to [\Str G,\V]$ is fully faithful (and w.l.o.g.~identity-on-homs) we obtain isomorphisms
\begin{eqnarray*}
\Str G\Alg^!(\sfE F J,A) & = & [\Str G,\V](\sfM \sfE FJ,\sfM A)\\
 & \cong & [\Str G,\V](\Str G(\tau_G J,-),\sfM A)\\
 & \cong & (\sfM A)\tau_G J\\
 & = & \C(jJ,U^{\Str G}A)
\end{eqnarray*}
that are $\V$-natural in $J \in \J$ and $A \in \Str G\Alg^!$ and witness that $\sfE F$ is a $j$-relative left adjoint for $U^{\Str G}:\Str G\Alg^! \to \C$.  The unit of this $j$-relative adjunction is obtained by chasing $1_{\sfE FJ}$ along these isomorphisms and so is precisely the unit of the representation \eqref{eq:repr}, namely $\eta_J:J \to GFJ = U^{\Str G}\sfE FJ$.
\end{proof}

\noindent We now establish the following intrinsic characterization of (strictly) $\J$-algebraic $\V$-categories:

\begin{theo}
\label{J_alg_char_thm}
Let $j : \J \hookrightarrow \C$ be an amenable subcategory of arities, and let $G : \A \to \C$ be a $\V$-functor. Then $G$ is a $\J$-algebraic $\V$-functor iff the following conditions are satisfied: 
\begin{enumerate}[leftmargin=*]
\item $G : \A \to \C$ has a $j$-relative left adjoint $F : \J \to \A$ \eqref{J_rel_adj}.
\item The fully faithful $\V$-functor $j_F : \A_F \rightarrowtail \A$ of \eqref{eq:A_F} is dense.
\item A presheaf $X : \A_F^\op \to \V$ is a $j_F$-nerve iff the presheaf $X \circ i_F^\op : \J^\op \to \V$ is a $j$-nerve. 
\end{enumerate}
Furthermore, $G$ is a strictly $\J$-algebraic $\V$-functor iff conditions (1)--(3) are satisfied as well as
\begin{enumerate}[leftmargin=*]
\item[4.] $G : \A \to \C$ is a discrete isofibration \eqref{discrete_iso}.
\end{enumerate}
\end{theo}

\begin{proof}
We may assume without loss of generality that (1) holds, since this condition holds if $G$ is $\J$-algebraic, by \ref{thm:basic_charn_jalgcats}.  In particular, $G$ is $\J$-tractable by \ref{adjointtractable}, and we can form the following commutative square in $\V'\CAT$ as in \ref{unitmorphism}
\begin{equation}\label{eq:algebraic_pb}
\begin{tikzcd}
	{\A} && {[\Str G, \V]} \\
	\\
	\C && {\left[\J^\op, \V\right]}
	\arrow["{\widetilde{m}_G}", from=1-1, to=1-3]
	\arrow["{\left[\tau_G, 1\right]}", from=1-3, to=3-3]
	\arrow["{G}"', from=1-1, to=3-1]
	\arrow["{N_j}"', from=3-1, to=3-3]
\end{tikzcd}
\end{equation}
where $\widetilde{m}_G$ is the transpose of $m_G:\Str G \to [\A,\V]$. By Proposition \ref{Str_prop}, we may take $\Str G = \A_F^\op$, $\tau_G = i_F^\op$, and $m_G \cong \A(j_F-,?):\A_F^\op \to [\A,\V]$. Consequently, the top side $\widetilde{m}_G$ of \eqref{eq:algebraic_pb} is isomorphic to the nerve $\V'$-functor $N_{j_F} = \A(j_F?,-):\A \to \left[\A_F^\op, \V\right]$ for $j_F:\A_F \to \A$. Therefore (2) holds iff $\widetilde{m}_G$ is fully faithful. Also, the essential image of $\widetilde{m}_G$ coincides with that of $N_{j_F}$ and so consists of precisely the $j_F$-nerves.

Using these observations, we now establish the given characterization of strictly $\J$-algebraic $\V$-functors, for which we can also assume (4) without loss of generality, in view of \ref{discrete_iso}. By Proposition \ref{unitinvertible}, $G$ is strictly $\J$-algebraic iff the unit morphism $\sfE_{(\A, G)} : (\A, G) \to \Sem\Str G$ of the structure--semantics adjunction \eqref{strsemadjunction} is invertible, iff (by the definition of $\sfE_{(\A, G)}$ in Definition \ref{unitmorphism}) the square \eqref{eq:algebraic_pb} is a pullback in $\V'\CAT$. Since $G$ and $[\tau_G, 1]$ are discrete isofibrations \eqref{discrete_iso} and $N_j : \C \to \left[\J^\op, \V\right]$ is fully faithful, we deduce from \cite[Lemma 14]{BourkeGarner} that the commutative square \eqref{eq:algebraic_pb} is a pullback iff (a) $\widetilde{m}_G$ is fully faithful, and (b) a presheaf $X : \Str G \to \V$ is in the essential image of $\widetilde{m}_G$ iff $X \circ \tau_G : \J^\op \to \V$ is in the essential image of $N_j : \C \to \left[\J^\op, \V\right]$, i.e.~is a $j$-nerve. But (a) is equivalent to (2) as noted above, while (b) is equivalent to (3) since the essential image of $\widetilde{m}_G$ consists of the $j_F$-nerves. Thus the desired characterization of strictly $\J$-algebraic $\V$-functors is proved.

Lastly we prove the first assertion of the Theorem.  If $G$ is $\J$-algebraic, then there is a strictly $\J$-algebraic $\V$-category $(\A', G')$ over $\C$ such that $\A \simeq \A'$ in the pseudo-slice 2-category $\V\CAT//\C$, so that conditions (1)--(3) hold for $G$, since they hold for $G'$ and are stable under equivalence in $\V\CAT//\C$. Conversely, suppose that $G$ satisfies conditions (1)--(3). Within the square \eqref{eq:algebraic_pb}, the top side $\widetilde{m}_G$ is fully faithful since (2) holds, while the right side $[\tau_G,1]$ has a faithful underlying ordinary functor, since $\tau_G$ is the identity on objects, so the common composite in \eqref{eq:algebraic_pb} has a faithful underlying ordinary functor. This entails that the left side $G$ has a faithful underlying ordinary functor. Hence, by Proposition \ref{discrete_iso_prop} there is a $\V$-category $\A'$ equipped with a discrete isofibration $G' : \A' \to \C$ and an equivalence $L:\A' \xrightarrow{\sim} \A$ such that $GL \cong G'$, from which it follows by \cite[2.5]{EAT} that $L$ underlies an equivalence $\A' \simeq \A$ in $\V\CAT//\C$.  Hence $G'$ also satisfies conditions (1)--(3), in addition to (4), so that $G'$ is strictly $\J$-algebraic, whence $G$ is $\J$-algebraic.     
\end{proof}

\section{The monad--theory equivalence}
\label{mnd_th_section}

In this section, we fix an arbitrary subcategory of arities $j : \J \hookrightarrow \C$. 	Under the assumption that $\J$ is amenable, we shall establish in Theorem \ref{strsemcorestricted} that the idempotent structure--semantics adjunction of Theorem \ref{strsemadjunction} yields an idempotent adjunction between admissible $\J$-pretheories and $\V$-monads on $\C$. We shall then show in Theorem \ref{monadtheoryequivalence} that this adjunction restricts to an equivalence between $\J$-theories and \emph{$\J$-nervous} $\V$-monads on $\C$. 

\begin{para}
\label{monad_para}
We write $\Mnd(\C)$ for the (ordinary) category of $\V$-monads on $\C$. For a $\V$-monad $\T$ on $\C$, we regard the $\V$-category $\T\Alg$ of $\T$-algebras\footnote{Note that $\T\Alg$ is indeed a $\V$-category because $\V$ has equalizers.} as a $\V$-category over $\C$ by means of the forgetful $\V$-functor $U^\T : \T\Alg \to \C$. We then say that a $\V$-functor $G : \A \to \C$ is \emph{strictly monadic}, or that $\A$ is a \emph{strictly monadic $\V$-category over $\C$}, if there is some $\V$-monad $\T$ on $\C$ such that $\A \cong \T\Alg$ in $\V\CAT/\C$. We write $\MONADIC^!(\C)$ for the full subcategory of $\V\CAT/\C$ consisting of the strictly monadic $\V$-categories over $\C$. There is a $\V$-monad \emph{semantics} functor $\ALG : \Mnd(\C)^\op \to \V\CAT/\C$ given by $\T \mapsto \T\Alg$, and this functor is fully faithful by \cite[Pages 74--75]{Dubucbook}. Corestricting $\ALG$ to its essential image thus yields an equivalence $\Mnd(\C)^\op \xrightarrow{\sim} \MONADIC^!(\C)$. A pseudo-inverse to this equivalence is obtained by associating to each strictly monadic $\V$-category $(\A, G)$ over $\C$ the $\V$-monad induced by a choice of left adjoint to $G$, and then the Eilenberg-Moore comparison isomorphisms witness that this functor $\MONADIC^!(\C) \to \Mnd(\C)^\op$ is indeed pseudo-inverse to the equivalence $\Mnd(\C)^\op \xrightarrow{\sim} \MONADIC^!(\C)$.   
\end{para} 

By definition, a \textit{weight} is a $\V$-functor $W : \B^\op \to \V$, where $\B$ is a (possibly large) $\V$-category.  

\begin{defn}
\label{Jflat}
{
A weighted colimit in $\C$ is \textbf{$\J$-stable} (cf.~\cite[Definition 6.1]{EAT}) if it is preserved by each $\C(J, -) : \C \to \V$ ($J \in \ob\J$). A weight $W : \B^\op \to \V$ is \textbf{$\J$-flat} (cf.~\cite[Definition 6.2]{EAT}) if all $W$-weighted colimits that exist in $\C$ are $\J$-stable. A \textbf{$\J$-flat colimit} in a $\V$-category $\A$ is a weighted colimit $W * D$ in $\A$ whose weight $W$ is $\J$-flat.
}
\end{defn}

\noindent Thus the property of $\J$-stability is defined only for colimits in $\C$, but we may consider $\J$-flat colimits in $\V$-categories other than $\C$, as $\J$-flatness is a property of the weight.

\begin{para}
\label{relativelystable}
{
By a \emph{weighted diagram} in a $\V$-category $\A$ we mean a pair $(W, D)$ consisting of a weight $W : \B^\op \to \V$ and a $\V$-functor $D : \B \to \A$. We say that a weighted diagram $(W, D)$ in $\C$ is \textbf{$\J$-stable} if every colimit $W \ast D$ that $\C$ admits is $\J$-stable. Given a $\V$-functor $G : \A \to \C$, we say that a weighted diagram $(W, D)$ in $\A$ is \textbf{$G$-relatively $\J$-stable} (cf.~\cite[Definition 6.1]{EAT}) if the weighted diagram $(W, GD)$ in $\C$ is $\J$-stable. 

Given a class $\Lambda$ of weighted diagrams and a $\V$-functor $G : \A \to \C$, we say that $G$ \emph{creates $\Lambda$-colimits} if for every weighted diagram $(W, D) \in \Lambda$ in $\A$ and every colimit cylinder\footnote{An object $C$ equipped with a colimit cylinder $\lambda$ in the sense of \cite[\S 3.1]{Kelly}.} $(C, \lambda)$ for $(W, GD)$, there is a unique cylinder $\left(\overline{C}, \overline{\lambda}\right)$ for $(W, D)$ with $\left(G\overline{C}, G\overline{\lambda}\right) = (C, \lambda)$, and furthermore $\left(\overline{C}, \overline{\lambda}\right)$ is a colimit cylinder for $(W, D)$. We use these concepts in the following proposition:  
}
\end{para}

\begin{prop}
\label{createscolimits}
Let $\scrT$ be an admissible $\J$-pretheory. Then the $\V$-functor $U^\scrT : \scrT\Alg^! \to \C$ creates $U^\scrT$-relatively $\J$-stable colimits (even $\V'$-enriched such colimits), and in particular creates $\J$-flat colimits.  
\end{prop}

\begin{proof}
Let us write $U := U^\scrT$. Let $\left(W : \B^\op \to \V, D : \B \to \scrT\Alg^!\right)$ be a $U$-relatively $\J$-stable weighted diagram in $\scrT\Alg^!$, and let $W \ast UD$ be a colimit of $(W, UD)$ in $\C$ with colimit cylinder $\lambda : W \Rightarrow \C\left(UD-, W \ast UD\right)$. We must show that there is a unique cylinder $\lambdabar : W \Rightarrow \scrT\Alg^!\left(D-, A\right)$ in $\scrT\Alg^!$ with $U\lambdabar = \lambda$, and that $\lambdabar$ is a colimit cylinder for $(W, D)$. Since $(W, D)$ is $U$-relatively $\J$-stable, the weighted colimit $W \ast UD$ is $\J$-stable, and hence is preserved by each $\C(J, -) : \C \to \V$ ($J \in \ob\J$), so $N_j : \C \to \left[\J^\op, \V\right]$ sends the colimit $W \ast UD$ to a pointwise colimit in $\left[\J^\op, \V\right]$. Since $\tau : \J^\op \to \scrT$ is identity-on-objects, it follows that the $\V'$-functor $[\tau, 1] : [\scrT, \V] \to \left[\J^\op, \V\right]$ creates pointwise colimits\footnote{I.e.~given $\V'$-functors $W : \B^\op \to \V$ and $D : \B \to [\scrT, \V]$, any pointwise colimit cylinder $(W \ast [\tau, 1]D, \lambda)$ that exists in $\left[\J^\op, \V\right]$ lifts uniquely to a cylinder in $[\scrT, \V]$, and the latter cylinder is a pointwise colimit cylinder.}. Then because the square \eqref{eqn:concrete_pb} defining $\scrT\Alg^!$ is a pullback and the fully faithful $\sfM^\scrT : \scrT\Alg^! \to [\scrT, \V]$ reflects colimits, the desired conclusion readily follows.       
\end{proof}

\begin{prop}
\label{strictlymonadic}
Let $\scrT$ be an admissible $\J$-pretheory. Then the $\V$-functor $U^\scrT : \scrT\Alg^! \to \C$ is strictly monadic. 
\end{prop}

\begin{proof}
The $\V$-functor $U^\scrT : \scrT\Alg^! \to \C$ has a left adjoint by admissibility of $\scrT$, and it creates $U^\scrT$-contractible coequalizers by Proposition \ref{createscolimits}, since contractible coequalizers are absolute colimits, and hence are $\J$-stable. So $U^\scrT$ is strictly monadic by the enriched Beck monadicity theorem \cite[Theorem II.2.1]{Dubucbook}. 
\end{proof}

\begin{rmk}
\label{monadic_rmk}
{
It follows from Proposition \ref{strictlymonadic} that a $\J$-pretheory $\scrT$ is admissible iff the $\V'$-functor $U^\scrT : \scrT\Alg^! \to \C$ has a left adjoint (i.e.~the requirement in Definition \ref{tractablearities} that $\scrT\Alg^!$ be a $\V$-category may be omitted, since it follows from the existence of such a left adjoint). For if the $\V'$-functor $U^\scrT : \scrT\Alg^! \to \C$ has a left adjoint, then $U^\scrT$ will be strictly $\V'$-monadic by the $\V'$-enriched version of Proposition \ref{strictlymonadic}, and hence $\scrT\Alg^!$ will be a $\V$-category because $\C$ is a $\V$-category. 
}
\end{rmk}

\noindent In view of Proposition \ref{strictlymonadic} we can now co-restrict the structure--semantics adjunction \eqref{strsemadjunction} to obtain the following theorem. The idempotent adjunction \eqref{eq:preth_mnd_adj} below generalizes the idempotent adjunction established by Bourke and Garner in \cite[Theorems 6 and 20]{BourkeGarner}; note that our proof relies on structure--semantics methods, whereas the proof given by Bourke and Garner employs different techniques (which are only available in their locally presentable setting).   

\begin{theo}
\label{strsemcorestricted}
Let $\J \hookrightarrow \C$ be an amenable subcategory of arities. Then the idempotent structure--semantics adjunction \eqref{strsemadjunction} co-restricts to an idempotent adjunction
\begin{equation}\label{eq:preth_mndc_adj}
\begin{tikzcd}
	{\Preth_{\underJ}^\sfa(\C)^\op} &&& {} & {\MONADIC^!(\C)}.
	\arrow[""{name=0, anchor=center, inner sep=0}, "\Sem", shift left=3, from=1-1, to=1-5]
	\arrow[""{name=1, anchor=center, inner sep=0}, "\Str", shift left=3, from=1-5, to=1-1]
	\arrow["\dashv"{anchor=center, rotate=90}, draw=none, from=1, to=0]
\end{tikzcd}
\end{equation}
Since $\MONADIC^!(\C) \simeq \Mnd(\C)^\op$, we then obtain an idempotent adjunction
\begin{equation}\label{eq:preth_mnd_adj}
\begin{tikzcd}
	{\Preth_{\underJ}^\sfa(\C)} &&& {} & {\Mnd(\C)},
	\arrow[""{name=0, anchor=center, inner sep=0}, "\sfm", shift left=3, from=1-1, to=1-5]
	\arrow[""{name=1, anchor=center, inner sep=0}, "\sft", shift left=3, from=1-5, to=1-1]
	\arrow["\vdash"{anchor=center, rotate=90}, draw=none, from=1, to=0]
\end{tikzcd}
\end{equation}
where $\sfm$ is the composite $\Preth_{\underJ}^\sfa(\C) \xrightarrow{\Sem^\op} \MONADIC^!(\C)^\op \xrightarrow{\sim} \Mnd(\C)$ and $\sft$ is the composite $\Mnd(\C) \xrightarrow{\sim} \MONADIC^!(\C)^\op \xrightarrow{\Str^\op} \Preth_{\underJ}^\sfa(\C)$. \qed  
\end{theo}

\begin{para}
\label{Kleisli_Jtheory}
{
Let $\J \hookrightarrow \C$ be an amenable subcategory of arities. Given an admissible $\J$-pretheory $\scrT$, the $\V$-monad $\sfm(\scrT)$ is the free concrete $\scrT$-algebra $\V$-monad, i.e.~the $\V$-monad on $\C$ induced by the adjunction $F^\scrT \dashv U^\scrT$. Conversely, given a $\V$-monad $\T$ on $\C$, $\sft(\T)$ is the $\J$-theory $\Str U^\T$, where $U^\T : \T\Alg \to \C$ is the forgetful $\V$-functor. Now $U^\T : \T\Alg \to \C$ has a $j$-relative left adjoint $F^\T j : \J \to \T\Alg$, which we can factor as an identity-on-objects $\V$-functor $i_\T$ followed by a fully faithful $\V$-functor $j_\T$ as in the diagram
\begin{equation}\label{eq:J_T}
\begin{tikzcd}
	{\J} && {\T\Alg}. \\
	& {\J_\T}
	\arrow["{Fj}", from=1-1, to=1-3]
	\arrow["{i_\T}"', from=1-1, to=2-2]
	\arrow["{j_\T}"', from=2-2, to=1-3]
\end{tikzcd}
\end{equation} 
By Proposition \ref{Str_prop}, we may take $\Str U^\T$ to be the $\J$-theory $\left(\J_\T^\op, i_\T^\op\right)$, which we may call the \textbf{Kleisli $\J$-theory of $\T$}. 
}
\end{para}

We know by Proposition \ref{unitinvertible} that a $\V$-monad $\T$ is fixed by the idempotent comonad determined by the adjunction $\sfm \dashv \sft$ of Theorem  \ref{strsemcorestricted} iff there is some admissible $\J$-pretheory $\scrT$ such that $\T$ is isomorphic to the free concrete $\scrT$-algebra $\V$-monad $\sfm(\scrT)$. We now show that these $\V$-monads admit an alternative characterization that is \emph{not} expressed in terms of pretheories. The following definition was given by Bourke and Garner \cite[Definition 17]{BourkeGarner} under more restrictive assumptions:

\begin{defn}
\label{Jnervous}
{
Let $j : \J \hookrightarrow \C$ be an amenable subcategory of arities. A $\V$-monad $\T$ on $\C$ is \textbf{$\J$-nervous} if the following conditions are satisfied:
\begin{enumerate}[leftmargin=*] 
\item The fully faithful $\V$-functor $j_\T : \J_\T \rightarrowtail \T\Alg$ of \ref{Kleisli_Jtheory} is dense. 
\item A presheaf $X : \J_\T^\op \to \V$ is a $j_\T$-nerve iff the composite presheaf $\J^\op \xrightarrow{\tau_\T} \J_\T^\op \xrightarrow{X} \V$ is a $j$-nerve.
\end{enumerate} 
We write $\Mnd_{\underJ}(\C)$ for the full subcategory of $\Mnd(\C)$ consisting of the $\J$-nervous $\V$-monads on $\C$. 
}
\end{defn} 

The following result generalizes the result \cite[Theorem 18]{BourkeGarner} of Bourke and Garner; while they give a direct proof in their setting, the proof that we give below employs our characterization theorem for strictly $\J$-algebraic $\V$-categories (Theorem \ref{J_alg_char_thm}).

\begin{prop}
\label{Jnervous_fixedpoint}
Let $j : \J \hookrightarrow \C$ be an amenable subcategory of arities, and let $\T$ be a $\V$-monad on $\C$. Then the following are equivalent: (1) $\T$ is $\J$-nervous; (2) $\T$ is fixed by the idempotent comonad $\sfm\sft$ determined by the adjunction $\sfm \dashv \sft$ of \eqref{eq:preth_mnd_adj}; (3) $\T\Alg$ is a strictly $\J$-algebraic $\V$-category over $\C$.
\end{prop}

\begin{proof}
In view of Theorem \ref{strsemcorestricted}, the counit morphism $\varepsilon_\T : \sfm(\sft(\T)) \to \T$ of the adjunction \eqref{eq:preth_mnd_adj} is invertible iff the unit morphism $\sfE_{\left(\T\Alg, U^\T\right)} : \left(\T\Alg, U^\T\right) \to \left(\J_\T^\op\Alg^!, U^{\J_\T^\op}\right)$ of the structure--semantics adjunction \eqref{strsemadjunction} is invertible, which is equivalent by Proposition \ref{unitinvertible} to $\left(\T\Alg, U^\T\right)$ being a strictly $\J$-algebraic $\V$-category over $\C$. Since $U^\T : \T\Alg \to \C$ is strictly monadic and thus a discrete isofibration \eqref{discrete_iso} and has the $j$-relative left adjoint $F^\T j : \J \to \T\Alg$, Theorem \ref{J_alg_char_thm} (in view of \ref{Kleisli_Jtheory}) then entails that $\left(\T\Alg, U^\T\right)$ is strictly $\J$-algebraic iff $\T$ is $\J$-nervous, as desired.  
\end{proof}

\noindent From Proposition  \ref{Jnervous_fixedpoint} we immediately deduce the following:

\begin{prop}
\label{Jnervous_equiv}
Let $\J \hookrightarrow \C$ be an amenable subcategory of arities. Then a $\V$-monad $\T$ on $\C$ is $\J$-nervous iff there is some admissible $\J$-pretheory $\scrT$ such that $\T$ is isomorphic to the free concrete $\scrT$-algebra $\V$-monad $\sfm(\scrT)$. \qed
\end{prop}

\noindent We now obtain the following useful property of $\J$-nervous $\V$-monads. Given a class $\Lambda$ of weighted diagrams and a $\V$-functor $G : \A \to \C$, we say that $G$ \emph{conditionally preserves $\Lambda$-colimits} if for each $(W, D) \in \Lambda$ in $\A$ with a colimit $W \ast D$, if $W \ast GD$ exists then $G$ preserves the colimit $W \ast D$ (see \cite[2.3]{EAT}). Note that if $G$ creates $\Lambda$-colimits, then $G$ conditionally preserves them (while in general $G$ need not preserve them, for lack of existence in $\C$). 

\begin{prop}
\label{JnervousJflat}
Let $\J \hookrightarrow \C$ be an amenable subcategory of arities, and let $\T = (T, \eta, \mu)$ be a $\J$-nervous $\V$-monad on $\C$. Then $T : \C \to \C$ conditionally preserves $\J$-flat colimits. 
\end{prop}

\begin{proof}
By Proposition \ref{Jnervous_equiv}, there is an admissible $\J$-pretheory $\scrT$ such that $T \cong U^\scrT F^\scrT$. The result now follows because $U^\scrT$ creates $\J$-flat colimits by Proposition \ref{createscolimits} and the left adjoint $F^\scrT$ preserves all colimits. 
\end{proof}

\begin{theo}
\label{monadtheoryequivalence}
Let $\J \hookrightarrow \C$ be an amenable subcategory of arities. Then the adjunction $\sfm \dashv \sft$ of \eqref{eq:preth_mnd_adj} restricts to an equivalence
\[\begin{tikzcd}
	{\Th_{\underJ}(\C)} &&& {} & {\Mnd_{\underJ}(\C)}
	\arrow[""{name=0, anchor=center, inner sep=0}, "\sfm", shift left=3, from=1-1, to=1-5]
	\arrow[""{name=1, anchor=center, inner sep=0}, "\sft", shift left=3, from=1-5, to=1-1]
	\arrow["\sim"{anchor=center}, draw=none, from=1, to=0]
\end{tikzcd} \]
between $\J$-theories and $\J$-nervous $\V$-monads on $\C$. Furthermore:
\begin{enumerate}[leftmargin=*]
\item The inclusion $\Mnd_{\underJ}(\C) \hookrightarrow \Mnd(\C)$ is coreflective, with coreflector $\sfm\sft : \Mnd(\C) \to \Mnd_{\underJ}(\C)$. 
\item For every $\J$-nervous $\V$-monad $\T$ we have $\sft(\T)\Alg^! = \J_\T^\op\Alg^! \cong \T\Alg$ in $\V\CAT/\C$, naturally in $\T \in \Mnd_{\underJ}(\C)$. 
\item For every admissible $\J$-pretheory $\scrT$ (and in particular, every $\J$-theory $\scrT$) we have $\sfm(\scrT)\Alg \cong \scrT\Alg^!$ in $\V\CAT/\C$, naturally in $\scrT \in \Preth_{\underJ}^\sfa(\C)$. 
\item For every admissible $\J$-pretheory $\scrT$ we have $\scrT\Alg^! \cong (\sft\sfm\scrT)\Alg^! \cong \left(\Str\Sem\scrT\right)\Alg^!$ in $\V\CAT/\C$, naturally in $\scrT \in \Preth_{\underJ}^\sfa(\C)$, where $\sft\sfm\,\scrT \cong \Str\Sem\scrT$ is the $\J$-theory reflection of $\scrT$ \eqref{theory_Jalg_equiv}.
\end{enumerate}   
\end{theo}

\begin{proof}
The adjunction restricts to an equivalence by Propositions \ref{idempotent_adjunction}(2) and \ref{Jnervous_fixedpoint}, while (1) follows from Proposition \ref{Jnervous_fixedpoint} and the idempotence of the adjunction \eqref{eq:preth_mnd_adj}, and (2) follows immediately from \ref{Kleisli_Jtheory} and Proposition \ref{Jnervous_fixedpoint}. Given an admissible $\J$-pretheory $\scrT$ (in particular, a $\J$-theory $\scrT$) and its $\J$-theory reflection $\Str\Sem\scrT$, we have $\scrT\Alg^! \cong \sfm(\scrT)\Alg$ by the definition of $\sfm(\scrT)$ and because $U^\scrT : \scrT\Alg^! \to \C$ is strictly monadic \eqref{strictlymonadic}, and these isomorphisms are natural in $\scrT \in \Preth_{\underJ}^\sfa(\C)$ (by the naturality of the Eilenberg-Moore comparison isomorphisms, \ref{monad_para}). This proves (3). Since the $\V$-monad $\sfm(\scrT)$ is $\J$-nervous \eqref{Jnervous_equiv} we deduce from (2) that $\scrT\Alg^! \cong \sfm(\scrT)\Alg \cong \sft(\sfm(\scrT))\Alg^! \cong \left(\Str\Sem\scrT\right)\Alg^!$ in $\V\CAT/\C$, yielding (4).    
\end{proof}

\noindent The following theorem now follows immediately from Theorems \ref{J_alg_char_thm} and \ref{monadtheoryequivalence}:

\begin{theo}
\label{Jnervous_thm}
Let $\J \hookrightarrow \C$ be an amenable subcategory of arities, and let $G : \A \to \C$ be a $\V$-functor. Then the following are equivalent: (1) $G$ is strictly $\J$-algebraic; (2) $G$ is strictly monadic and the associated $\V$-monad on $\C$ is $\J$-nervous; (3) $G$ satisfies conditions (1)--(4) of Theorem \ref{J_alg_char_thm}. \qed
\end{theo}

\noindent We write $\ALG : \Mnd_{\underJ}(\C)^\op \to \V\CAT/\C$ for the restriction of $\ALG : \Mnd(\C)^\op \to \V\CAT/\C$. For a strongly amenable subcategory of arities $\J \hookrightarrow \C$, we write $\ALG^! : \Preth_{\underJ}(\C)^\op \to \V\CAT/\C$ for the composite $\Preth_{\underJ}(\C)^\op \xrightarrow{\Sem} \J\Tract(\C) \hookrightarrow \V\CAT/\C$.  By \ref{idempotent_adjunction}(3), the restriction of $\ALG^!$ to $\Th_{\underJ}(\C)^\op$ is a fully faithful functor that we denote simply by $\ALG^! : \Th_{\underJ}(\C)^\op \to \V\CAT/\C$. The following result generalizes certain results that Bourke and Garner proved in their locally presentable context (see \cite[\S 5.4 and Proposition 31]{BourkeGarner}). 

\begin{prop}
\label{cocompletecor}
Suppose that $\V$ is complete and cocomplete and that $\C$ is complete, and let $\J \hookrightarrow \C$ be a small and strongly amenable subcategory of arities. Then $\Preth_{\underJ}(\C)$, $\Th_{\underJ}(\C)$, and $\Mnd_{\underJ}(\C)$ are cocomplete. Moreover, small colimits in these categories are algebraic, i.e.~are sent to limits by the functors $\ALG^! : \Preth_{\underJ}(\C)^\op \to \V\CAT/\C$, $\ALG^! : \Th_{\underJ}(\C)^\op \to \V\CAT/\C$, and $\ALG : \Mnd_{\underJ}(\C)^\op \to \V\CAT/\C$. 
\end{prop}

\begin{proof}
If $\Preth_{\underJ}(\C)$ is cocomplete, then since its full subcategory $\Th_{\underJ}(\C)$ is reflective by Proposition \ref{theory_Jalg_equiv}, $\Th_{\underJ}(\C)$ will also be cocomplete, and thus $\Mnd_{\underJ}(\C) \simeq \Th_{\underJ}(\C)$ will be cocomplete. Since $\V$ is cocomplete, so is the category $\V\Cat$ of small $\V$-categories (by \cite[Corollary 2.14]{WolffVcat}), and hence so is the co-slice category $\J^\op/\V\Cat$ (by, e.g.,~\cite[Proposition 2.16.3]{Borceux1}). Now $\Preth_{\underJ}(\C)$ is equivalent to the full subcategory of $\J^\op/\V\Cat$ consisting the $\V$-functors that are bijective-on-objects, and the bijective-on-objects $\V$-functors form the left class of a factorization system on $\V\Cat$, so that $\Preth_{\underJ}(\C)$ is coreflective in the cocomplete category $\J^\op/\V\Cat$ by (e.g.) \cite[2.12]{Carboni_localization}, and hence $\Preth_{\underJ}(\C)$ is cocomplete.

The functor $\ALG^! : \Th_{\underJ}(\C)^\op \to \V\CAT/\C$ preserves small limits, as it is the composite $\Th_{\underJ}(\C)^\op \xrightarrow{\Sem} \J\Tract(\C) \hookrightarrow \V\CAT/\C$, whose first factor preserves limits as a right adjoint (see Theorem \ref{strsemadjunction} and Theorem \ref{theory_Jalg_equiv}), and whose second factor preserves small limits by Proposition \ref{JTract_lims}. In view of Theorem \ref{monadtheoryequivalence}, the functor $\ALG^! : \Preth_{\underJ}(\C)^\op \to \V\CAT/\C$ is isomorphic to the composite $\Preth_{\underJ}(\C)^\op \xrightarrow{\Str\Sem} \Th_{\underJ}(\C)^\op \xrightarrow{\ALG^!} \V\CAT/\C$. The first factor is a right adjoint (being the opposite of a left adjoint, \ref{theory_Jalg_equiv}) and hence preserves limits, while the second factor preserves small limits, as already shown. Finally, the functor $\ALG : \Mnd_{\underJ}(\C)^\op \to \V\CAT/\C$ is isomorphic to the composite $\Mnd_{\underJ}(\C)^\op \xrightarrow{\sim} \Th_{\underJ}(\C)^\op \xrightarrow{\ALG^!} \V\CAT/\C$ by Theorem \ref{monadtheoryequivalence}, where the second factor preserves small limits (as already shown) and the first clearly does (as an equivalence).  
\end{proof}

\noindent We conclude \S\ref{mnd_th_section} by showing that any small subcategory of arities that is contained in some strongly amenable subcategory of arities is itself strongly amenable (Theorem \ref{pushoutthm} below). 

\begin{para}
\label{pushoutfunctor}
{
Suppose that $\V$ is cocomplete. Let $k : \scrK \hookrightarrow \C$ be a subcategory of arities that is contained in some small subcategory of arities $j : \J \hookrightarrow \C$, and let $i : \scrK \hookrightarrow \J$ be the inclusion. As mentioned in Definition  \ref{Jtheory}, the category $\Preth_{\scrK}(\C)$ of $\scrK$-pretheories is the full subcategory of $\scrK^\op/\V\Cat$ consisting of the identity-on-objects $\V$-functors, and similarly for $\Preth_{\underJ}(\C)$. We have a functor $i^\ast : \J^\op/\V\Cat \to \scrK^\op/\V\Cat$ given by precomposition with $i^\op : \scrK^\op \to \J^\op$, and $i^\ast$ is readily seen to have a left adjoint given by pushout along $i^\op : \scrK^\op \to \J^\op$ (since $\V\Cat$ is cocomplete by \cite[Corollary 2.14]{WolffVcat}). If $\scrT$ is a $\scrK$-pretheory, then the pushout of $\tau : \scrK^\op \to \scrT$ along $i^\op : \scrK^\op \to \J^\op$ can be chosen to be identity-on-objects, since bijective-on-objects $\V$-functors are stable under pushout, as they constitute the left class of a factorization system on $\V\Cat$. So the left adjoint of $i^\ast$ restricts to a functor $i_\ast : \Preth_\scrK(\C) \to \Preth_{\underJ}(\C)$ given by pushout along $i^\op : \scrK^\op \to \J^\op$.      
}
\end{para}

\begin{lem}
\label{pushoutalgebras}
Suppose that $\V$ is complete and cocomplete, and let $k : \scrK \hookrightarrow \C$ be a subcategory of arities that is contained in some small subcategory of arities $j : \J \hookrightarrow \C$. Then for each $\scrK$-pretheory $\scrT$ we have $\scrT\Alg^! \cong i_\ast(\scrT)\Alg^!$ in $\V\CAT/\C$.
\end{lem}

\begin{proof}
The $\J$-pretheory $i_\ast(\scrT)$ is obtained as the following pushout in $\V\Cat$:
\[\begin{tikzcd}
	{\scrK^\op} && {\J^\op} \\
	\\
	\scrT && {i_\ast(\scrT)}.
	\arrow["{i^\op}", from=1-1, to=1-3]
	\arrow["\tau"', from=1-1, to=3-1]
	\arrow[from=3-1, to=3-3]
	\arrow["\tau_\J", from=1-3, to=3-3]
\end{tikzcd}\]
By applying the functor $[-, \V] : \V\Cat^\op \to \V\CAT$ to this pushout square, we then obtain the following pullback square in $\V\CAT$:
\[\begin{tikzcd}
	{\left[i_\ast(\scrT), \V\right]} && {[\scrT, \V]} \\
	\\
	{\left[\J^\op, \V\right]} && {\left[\scrK^\op, \V\right]}.
	\arrow[from=1-1, to=1-3]
	\arrow["{[\tau_\J, 1]}"', from=1-1, to=3-1]
	\arrow["{[\tau, 1]}", from=1-3, to=3-3]
	\arrow["{\left[i^\op, 1\right]}"', from=3-1, to=3-3]
\end{tikzcd}\]
Composing this pullback square with the pullback square that defines $i_\ast(\scrT)\Alg^!$, we then obtain the following commutative diagram in $\V\CAT$ whose outer rectangle is a pullback:
\[\begin{tikzcd}
	{i_\ast(\scrT)\Alg^!} && {\left[i_\ast(\scrT), \V\right]} && {[\scrT, \V]} \\
	\\
	\C && {\left[\J^\op, \V\right]} && {\left[\scrK^\op, \V\right]}.
	\arrow[from=1-3, to=1-5]
	\arrow["{\left[\tau_\J, 1\right]}"', from=1-3, to=3-3]
	\arrow["{[\tau, 1]}", from=1-5, to=3-5]
	\arrow["{\left[i^\op, 1\right]}"', from=3-3, to=3-5]
	\arrow["{N_j}"', from=3-1, to=3-3]
	\arrow["{\sfM^{i_\ast(\scrT)}}", from=1-1, to=1-3]
	\arrow["{U^{i_\ast(\scrT)}}"', from=1-1, to=3-1]
\end{tikzcd}\]
But since the lower composite is equal to $N_k : \C \to \left[\scrK^\op, \V\right]$, we deduce that $\scrT\Alg^! \cong i_\ast(\scrT)\Alg^!$ in $\V\CAT/\C$ by the definition of $\scrT\Alg^!$ as a pullback.  
\end{proof}

\noindent From Lemma \ref{pushoutalgebras} we immediately deduce the following theorem:

\begin{theo}
\label{pushoutthm}
Suppose that $\V$ is complete and cocomplete, and let $\scrK \hookrightarrow \C$ be a subcategory of arities that is contained in some small and strongly amenable subcategory of arities. Then $\scrK \hookrightarrow \C$ is strongly amenable. \qed   
\end{theo}

\section{Examples}
\label{amenableexamples}

In this section, we shall establish some classes of examples of amenable and strongly amenable subcategories of arities. 

\subsection{Eleutheric subcategories of arities}
\label{eleuthericsubsection}

We shall first show that every (possibly large) \emph{eleutheric} \cite{EAT, Pres} subcategory of arities is amenable. Throughout this subsection,  as in \S\ref{strsemsection} and \S\ref{mnd_th_section}, we let $j : \J \hookrightarrow \C$ be a (possibly large) subcategory of arities in an arbitrary $\V$-category $\C$ enriched in a symmetric monoidal closed category $\V$ with equalizers. 

\begin{defn_sub}
\label{eleutheric}
{
The subcategory of arities $j : \J \hookrightarrow \C$ is \textbf{eleutheric} \cite{EAT, Pres} if every $\V$-functor $H : \J \to \C$ has a left Kan extension along $j$ that is preserved by each $\C(J, -) : \C \to \V$ ($J \in \ob\J$). Writing $\Phi_{\underJ}$ for the class of (possibly large) weights $\C(j-, C) : \J^\op \to \V$ ($C \in \ob\C$), we have that $\J$ is eleutheric iff $\C$ is $\Phi_{\underJ}$-cocomplete and each $\C(J, -) : \C \to \V$ ($J \in \ob\J$) preserves $\Phi_{\underJ}$-colimits \cite[Definition 3.3]{Pres}. The repletion of the full sub-$\V'$-category $\Phi_{\underJ} \hookrightarrow \left[\J^\op, \V\right]$ is the $\V$-category $j\Ner(\V)$. 
}
\end{defn_sub}

\noindent We shall show in Corollary  \ref{JnervousJarycor} below that if $\J$ is eleutheric, then a $\V$-monad on $\C$ is $\J$-nervous iff it is \emph{$\J$-ary} in the sense of the following definition:

\begin{defn_sub}
\label{Jary}
{
A $\V$-endofunctor $T : \C \to \C$ is \textbf{$\J$-ary} \cite{EAT, Pres} (or \textbf{$j$-ary}) if $T$ preserves $\Phi_{\underJ}$-colimits, or equivalently if $T$ preserves left Kan extensions along $j$. A $\V$-monad $\T$ on $\C$ is \textbf{$\J$-ary} if its underlying $\V$-endofunctor is $\J$-ary.
} 
\end{defn_sub}

\begin{egg_sub}
\label{eleuthericexamples}
{
We have the following examples of eleutheric subcategories of arities $\J \hookrightarrow \C$ and their associated $\J$-ary $\V$-endofunctors, in view of \cite[Examples 3.9 and 4.7]{Pres}:
\begin{enumerate}
[leftmargin=0pt,labelindent=0pt,itemindent=*,label=(\arabic*)]
\item Suppose that $\V$ is locally $\alpha$-presentable as a closed category (in the sense of \cite[7.4]{Kellystr}). If $\C$ is a locally $\alpha$-presentable $\V$-category and $\C_\alpha$ is a skeleton of the full sub-$\V$-category of $\C$ consisting of the (enriched) $\alpha$-presentable objects, then $\C_\alpha \hookrightarrow \C$ is a small and eleutheric subcategory of arities, and the $\C_\alpha$-ary $\V$-endofunctors are precisely the $\V$-endofunctors that preserve small conical $\alpha$-filtered colimits. In fact, \emph{every} small subcategory of arities in this locally presentable setting is contained in some small and eleutheric subcategory of arities \cite[Example 3.9(1)]{Pres}. When $\C = \V$, the subcategory of arities $\V_\alpha \hookrightarrow \V$ is a \emph{system of arities} \eqref{system_arities_para}, because the unit object $I$ is $\alpha$-presentable and the $\alpha$-presentable objects are closed under the monoidal product \cite[Example 3.4]{EAT}.  

In particular, when $\C = \V = \Set$, the system of arities $\FinCard \hookrightarrow \Set$ consisting of the finite cardinals is eleutheric, and the $\FinCard$-ary endofunctors are precisely the finitary endofunctors. 

\item Suppose that $\V$ has (conical) finite copowers $n \cdot I$ ($n \in \N$) of the unit object $I$. Then there is a small system of arities \eqref{system_arities_para} $j : \SF(\V) \rightarrowtail \V$, where $\SF(\V)$ is the $\V$-category with $\ob\left(\SF(\V)\right) = \N$ and $\SF(\V)(n, m) = \V(n \cdot I, m \cdot I)$ for all $n, m \in \N$ (see \cite[Example 3.7]{EAT}). Provided that $\V$ is complete, cocomplete, and cartesian closed, or more generally a $\pi$\emph{-category} in the sense of \cite{BorceuxDay}, the system of arities $j : \SF(\V) \rightarrowtail \V$ is eleutheric. In the case where $\V$ is cartesian closed, $\SF(\V)$ is isomorphic to the free $\V$-category on $\FinCard$ by \cite[\S 3]{KellyLackstronglyfinitary}, and the $\SF(\V)$-ary $\V$-endofunctors are precisely the \emph{strongly finitary} $\V$-endofunctors of Kelly and Lack \cite[\S 3]{KellyLackstronglyfinitary}.

\item The inclusion $\{I\} \hookrightarrow \V$ of the unit object is a small and eleutheric system of arities (see also \cite[Example 3.6]{EAT}), and the $\{I\}$-ary $\V$-endofunctors on $\V$ are precisely those $\V$-endofunctors that are isomorphic to $X \tensor (-) : \V \to \V$ for some $X \in \ob\V$.

\item Given any $\V$-category $\C$, the identity $\V$-functor $\C \hookrightarrow \C$ is an eleutheric subcategory of arities, and the $\C$-ary $\V$-endofunctors are arbitrary $\V$-endofunctors. When $\C = \V$, the subcategory of arities $\V \hookrightarrow \V$ is a system of arities \eqref{system_arities_para} (see also \cite[Example 3.5]{EAT}).  

\item If $\V$ is complete and $\A$ is a small $\V$-category, then the Yoneda embedding $\y_\A : \A^\op \rightarrowtail [\A, \V]$ is a small and eleutheric subcategory of arities, and a $\V$-endofunctor on $[\A, \V]$ is $\y_\A$-ary iff it is cocontinuous.

\item Suppose that $\V$ is complete and cocomplete, and let $\Phi$ be a class of small weights that satisfies Axiom A of Lack-Rosick\'y \cite{LR} and is \textit{locally small} in the sense of \cite[Definition 8.10]{KS} (as in \cite[p. 370]{LR}).  Let $\calT$ be a $\Phi$\emph{-theory}, i.e.~a small $\V$-category $\calT$ with $\Phi$-limits, and let $\C = \Phi\Mod(\calT)$ be the $\V$-category of \emph{models} of $\calT$ in $\V$, i.e.~the full sub-$\V$-category of $[\calT, \V]$ consisting of the $\Phi$-continuous $\V$-functors. As in \cite{LR}, we call such $\V$-categories $\C$ \emph{locally $\Phi$-presentable}. The (corestricted) Yoneda embedding $\y_\Phi:\calT^\op \rightarrowtail \C$ is then an eleutheric subcategory of arities, and a $\V$-endofunctor on $\C = \Phi\Mod(\calT)$ is $\y_\Phi$-ary iff it preserves small $\Phi$-flat colimits, so that the $\y_\Phi$-ary $\V$-monads on $\C = \Phi\Mod(\calT)$ are the \emph{$\Phi$-accessible} $\V$-monads of \cite{LR}. 

\item As a special case of (6), let $\mathbb{D}$ be any small class of small categories that is a \emph{sound doctrine} in the sense of \cite{ABLR}, and suppose that $\V$ is locally $\bbD$-presentable as a $\tensor$-category in the sense of \cite[\S 5.4]{LR}. Then we can take $\Phi := \Phi_{\mathbb{D}}$ to be the saturation of the class of (small) weights for conical $\mathbb{D}$-limits and cotensors by $\bbD$-presentable objects of $\V$. A $\V$-endofunctor on $\Phi_{\mathbb{D}}\Mod(\calT)$ is then $\y_{\Phi_{\mathbb{D}}}$-ary iff it preserves small $\Phi_{\mathbb{D}}$-flat colimits, which is equivalent to its preservation of small conical $\bbD$-filtered colimits \cite[Lemma 4.8]{Pres}. 
\end{enumerate}
}
\end{egg_sub}

We now wish to characterize eleutheric subcategories of arities in terms of free cocompletions for classes of weights. Recall from \cite[Remarks 3.7]{KS} and \cite[Theorem 5.35]{Kelly} that if $\Phi$ is a class of weights, then a $\V$-functor $F : \A \to \B$ \textbf{presents $\B$ as a free $\Phi$-cocompletion of $\A$} if $\B$ is $\Phi$-cocomplete and for any $\Phi$-cocomplete $\V$-category $\X$, composition with $F$ induces an equivalence of categories $\Phi\Cocts(\B, \X) \xrightarrow{\sim} \V\CAT(\A, \X)$, where $\Phi\Cocts(\B, \X)$ is the full subcategory of $\V\CAT(\B, \X)$ consisting of the $\Phi$-cocontinuous $\V$-functors. 

\begin{rmk_sub}
\label{smalleleutheric}
We now recall the following equivalent characterizations of \emph{small} eleutheric subcategories of arities obtained in \cite[Propositions 3.6 and 3.8]{Pres}, when $\V$ is complete and cocomplete:
\begin{enumerate}[leftmargin=*]
\item $\J$ is eleutheric iff $j : \J \hookrightarrow \C$ presents $\C$ as a free $\Phi_{\underJ}$-cocompletion of $\J$.

\item $\J$ is eleutheric iff there is a class of small weights $\Psi$ such that $j : \J \hookrightarrow \C$ presents $\C$ as a free $\Psi$-cocompletion of $\J$ (in which case a $\V$-endofunctor on $\C$ is $\J$-ary iff it is $\Psi$-cocontinuous \cite[Proposition 4.2]{Pres}).
\end{enumerate}
The first author showed in \cite[Theorem 7.8]{EAT} that if $\V$ is only assumed to have equalizers and $j : \J \hookrightarrow \V$ is a (possibly large) \emph{system} of arities in $\V$ \eqref{system_arities_para}, then $\J$ is eleutheric iff $j : \J \hookrightarrow \V$ presents $\V$ as a free $\Phi_{\underJ}$-cocompletion of $\J$. We now wish to simultaneously generalize the latter result  and (1) above to the more general context of the present subsection. We first require the following result, which is  a generalization of \cite[Proposition 7.9]{EAT}.
\end{rmk_sub}

\begin{prop_sub}
\label{leftKanJary}
Let $\X$ be a $\Phi_{\underJ}$-cocomplete $\V$-category, and let $H : \C \to \X$ be a $\V$-functor. If $H$ preserves $\Phi_{\underJ}$-colimits, then $H$ is a left Kan extension along $j : \J \hookrightarrow \C$, and the converse also holds if $\J$ is eleutheric. 
\end{prop_sub}

\begin{proof}
We first deduce by \cite[Theorem 5.29]{Kelly} that $H$ is a left Kan extension along $j : \J \hookrightarrow \C$ iff $H$ preserves the weighted colimits $C = \C(j-, C) \ast j$ ($C \in \ob\C$) that exhibit $j$ as a dense $\V$-functor, since these weighted colimits constitute a \emph{density presentation} for $j$ (see \cite[\S 5.4]{Kelly} and also \ref{density_pres} below). So if $H$ preserves $\Phi_{\underJ}$-colimits, then $H$ certainly preserves the given weighted colimits and is thus a left Kan extension along $j$. Supposing now that $\J$ is eleutheric, let $F : \J \to \X$ be a $\V$-functor, and let us show that $\Lan_j F : \C \to \X$ (which exists because $\X$ is $\Phi_{\underJ}$-cocomplete) preserves $\Phi_{\underJ}$-colimits. Recall that $j\Ner(\V)$ is a $\V$-category and may be described as the repletion of $\Phi_{\underJ} \hookrightarrow \left[\J^\op, \V\right]$ \eqref{eleutheric}. Since $\X$ is $\Phi_{\underJ}$-cocomplete, the $\V$-functor $F : \J \to \X$ induces a $\V$-functor $(-) \ast F : j\Ner(\V) \to \X$, and $\Lan_j F$ then factors as the composite $\C \xrightarrow{N_j} j\Ner(\V) \xrightarrow{(-) \ast F} \X$. The eleuthericity of $\J$ entails that the equivalence $N_j$ sends each $\Phi_{\underJ}$-colimit in $\C$ to a \emph{pointwise} colimit. But $(-) \ast F$ sends pointwise colimits to colimits in $\X$ \cite[(3.23)]{Kelly}, so that the composite $\Lan_jF$ preserves $\Phi_{\underJ}$-colimits.   
\end{proof}

\begin{cor_sub}
Suppose that $j:\J \hookrightarrow \C$ is eleutheric.  Then a $\V$-monad $\T$ on $\C$ is $\J$-ary iff its underlying $\V$-endofunctor $T:\C \to \C$ is a left Kan extension along $j$.
\end{cor_sub}

\noindent Given a $\Phi_{\underJ}$-cocomplete $\V$-category $\X$, we write $\mathscr{L}_\X$ for the full subcategory of $\V\CAT(\C, \X)$ consisting of those $\V$-functors that are left Kan extensions along $j$. 

\begin{lem_sub}
\label{cocompletion_lem}
$j : \J \hookrightarrow \C$ presents $\C$ as a free $\Phi_{\underJ}$-cocompletion of $\J$ iff $\scrL_\X = \Phi_{\underJ}\Cocts(\C, \X)$ for each $\Phi_{\underJ}$-cocomplete $\V$-category $\X$. 
\end{lem_sub}

\begin{proof}
For each $\Phi_{\underJ}$-cocomplete $\V$-category $\X$, we have an adjunction 
\[\begin{tikzcd}
	{\V\CAT(\J, \X)} &&& {} & {\V\CAT(\C, \X)}
	\arrow[""{name=0, anchor=center, inner sep=0}, "{\Lan_j}", shift left=3, from=1-1, to=1-5]
	\arrow[""{name=1, anchor=center, inner sep=0}, "{\V\CAT(j, \X)}", shift left=3, from=1-5, to=1-1]
	\arrow["\vdash"{anchor=center, rotate=90}, draw=none, from=1, to=0]
\end{tikzcd} \]
with $\Lan_j$ fully faithful because $j$ is fully faithful. The functor $\V\CAT(j, \X)$ therefore restricts to an equivalence $E_1 : \scrL_\X \xrightarrow{\sim} \V\CAT(\J, \X)$. 
Noting that $\Phi_{\underJ}\Cocts(\C, \X) \subseteq \scrL_\X$ by Proposition \ref{leftKanJary}, consider the following commutative diagram, where $E_2 : \Phi_{\underJ}\Cocts(\C, \X) \to \V\CAT(\J, \X)$ is defined to be the given composite:
\[\begin{tikzcd}
	{\Phi_{\underJ}\Cocts(\C, \X)} \\
	{\mathscr{L}_\X} && {\V\CAT(\J, \X)}.
	\arrow[hook, from=1-1, to=2-1]
	\arrow["{E_2}", from=1-1, to=2-3]
	\arrow["{E_1}"', from=2-1, to=2-3]
\end{tikzcd}\]
Since the inclusion $\Phi_{\underJ}\Cocts(\C, \X) \hookrightarrow \scrL_\X$ is replete and $E_1$ is an equivalence, one readily verifies that $E_2$ is an equivalence iff $\Phi_{\underJ}\Cocts(\C, \X) = \scrL_\X$. But $j : \J \hookrightarrow \C$ presents $\C$ as a free $\Phi_{\underJ}$-cocompletion of $\J$ iff $E_2 : \Phi_{\underJ}\Cocts(\C, \X) \to \V\CAT(\J, \X)$ is an equivalence for each $\Phi_{\underJ}$-cocomplete $\V$-category $\X$, whence the desired result follows.   
\end{proof}

\begin{prop_sub}
\label{eleuthericcocompletion}
If $\J$ is eleutheric, then $j : \J \hookrightarrow \C$ presents $\C$ as a free $\Phi_{\underJ}$-cocompletion of $\J$. If $\V$ is $\Phi_{\underJ}$-cocomplete, then the converse also holds. 
\end{prop_sub}

\begin{proof}
If $\J$ is eleutheric, then $\C$ is $\Phi_{\underJ}$-cocomplete and $\mathscr{L}_\X = \Phi_{\underJ}\Cocts(\C, \X)$ for each $\Phi_{\underJ}$-cocomplete $\V$-category $\X$ by Proposition \ref{leftKanJary}, so that $j : \J \hookrightarrow \C$ presents $\C$ as a free $\Phi_{\underJ}$-cocompletion of $\J$ by Lemma \ref{cocompletion_lem}. Conversely, suppose that $\V$ is $\Phi_{\underJ}$-cocomplete and that $j : \J \hookrightarrow \C$ presents $\C$ as a free $\Phi_{\underJ}$-cocompletion of $\J$. Then $\mathscr{L}_\V = \Phi_{\underJ}\Cocts(\C, \V)$ by Lemma \ref{cocompletion_lem}. But each $\C(J, -) : \C \to \V$ ($J \in \ob\J$) is a left Kan extension of $\J(J, -) : \J \to \V$ along $j$ and hence is $\Phi_{\underJ}$-cocontinuous, proving that $\J$ is eleutheric.     
\end{proof} 

\noindent We now show that eleutheric subcategories of arities are amenable. Given a $\J$-theory $\scrT$, recall from \ref{Jtheoryfunctor} and Remark \ref{phi_rmk} that the $\V'$-functor $\phi_\scrT : \J \to \scrT\Alg^!$ is a $j$-relative left adjoint for $U^\scrT : \scrT\Alg^! \to \C$. 

\begin{theo_sub}
\label{eleuthericamenable}
Every eleutheric subcategory of arities $\J \hookrightarrow \C$ is amenable.  
\end{theo_sub}

\begin{proof}
Let $\scrT$ be a $\J$-theory. By Remark \ref{monadic_rmk}, it suffices to show that the $\V'$-functor $U^\scrT : \scrT\Alg^! \to \C$ has a left adjoint. Since $\phi_\scrT : \J \to \scrT\Alg^!$ is a $j$-relative left adjoint for $U^\scrT : \scrT\Alg^! \to \C$, it is equivalent by \cite[Lemma 8.4]{EAT} to show that the $\V'$-enriched left Kan extension of $\phi_\scrT : \J \to \scrT\Alg^!$ along $j : \J \to \C$ exists. It therefore suffices to show for each $C \in \ob\C$ that $\scrT\Alg^!$ admits the weighted colimit $\C(j-, C) \ast \phi_\scrT$. By (the $\V'$-enriched version of) Proposition \ref{createscolimits}, it then suffices to show that $\C$ admits the weighted colimit $\C(j-, C) \ast U^\scrT\phi_\scrT$, and that this colimit is $\J$-stable; but since $U^\scrT\phi_\scrT = T_\scrT : \J \to \C$ \eqref{Jtheoryfunctor}, this follows from $\J$ being eleutheric.    
\end{proof}

\noindent From Theorems \ref{monadtheoryequivalence} and \ref{eleuthericamenable}, we thus obtain an equivalence between $\J$-theories and $\J$-nervous $\V$-monads on $\C$ for each eleutheric subcategory of arities $\J \hookrightarrow \C$. We now proceed to show (see Corollary \ref{JnervousJarycor} below) that if $\J$ is eleutheric, then $\J$-nervous $\V$-monads coincide with $\J$-ary $\V$-monads \eqref{Jary}. 

\begin{lem_sub}
\label{monadisolemma}
Let $G : \A \to \X$ and $H : \B \to \X$ be strictly monadic $\V$-functors, let $F$ be left adjoint to $G$ with unit $\eta : 1 \Rightarrow GF$, and let $K : \A \to \B$ be a $\V$-functor with $HK = G$. Suppose that $\eta : 1 \Rightarrow GF = HKF$ exhibits $KF$ as left adjoint to $H$. Then $K : \A \to \B$ is an isomorphism.  
\end{lem_sub}

\begin{proof}
By strict monadicity of $G$ and $H$, we may suppose without loss of generality that $(\A, G) = \left(\T\Alg, U^\T\right)$ and $(\B, H) = \left(\bbS\Alg, U^\bbS\right)$ for $\V$-monads $\T = (T, \eta, \mu)$ and $\bbS = (S, \zeta, \nu)$ on $\X$, so that $K = \alpha^* : \T\Alg \to \bbS\Alg$ for a unique $\V$-monad morphism $\alpha : \bbS \to \T$ (see the first paragraph of \S\ref{mnd_th_section}). Now $\alpha^*$ sends each free $\T$-algebra $(TX, \mu_X)$ ($X \in \ob\X$) to the $\bbS$-algebra $(TX, \mu_X \circ \alpha_{TX})$, and the fact that $\alpha : \bbS \to \T$ is a $\V$-monad morphism entails that $\alpha_X : SX \to TX$ is an $\bbS$-algebra morphism $\alpha_X : (SX, \nu_X) \to (TX, \mu_X \circ \alpha_{TX})$ that commutes with the unit components $\zeta_X : X \to SX$ and $\eta_X : X \to TX$. But $\eta_X : X \to TX$ exhibits $(TX, \mu_X \circ \alpha_{TX})$ as a free $\bbS$-algebra on $X$ by assumption, so that $\alpha_X : (SX, \nu_X) \to (TX, \mu_X \circ \alpha_{TX})$ is an isomorphism (being the canonical isomorphism between two free $\bbS$-algebras on $X$). So $\alpha : \bbS \to \T$ is an isomorphism, and hence $K = \alpha^*$ is an isomorphism.     
\end{proof}

\begin{lem_sub}
\label{monad_iso_lemma_2}
Let $G : \A \to \C$ and $H : \B \to \C$ be strictly monadic $\V$-functors, let $F$ be left adjoint to $G$ with unit $\eta : 1 \Rightarrow GF$, and let $K : \A \to \B$ be a $\Phi_{\underJ}$-cocontinuous $\V$-functor with $HK = G$. Suppose that $\eta j : j \Rightarrow GFj = HKFj$ exhibits $KFj$ as a $j$-relative left adjoint for $H$ \eqref{J_rel_adj}. Then $K : \A \to \B$ is an isomorphism. 
\end{lem_sub}

\begin{proof}
To show that $K$ is an isomorphism, it suffices by Lemma \ref{monadisolemma} to show that $\eta : 1_\C \Rightarrow GF = HKF$ exhibits $KF$ as left adjoint to $H$. Since $1 : Fj \Rightarrow Fj$ exhibits $F$ as $\Lan_j Fj$ \eqref{J_rel_adj} and $K : \A \to \B$ preserves left Kan extensions along $j$, we deduce that $1 : KFj \Rightarrow KFj$ exhibits $KF$ as $\Lan_j KFj$. By assumption, $\eta j : j \Rightarrow GFj = HKFj$ exhibits $KFj$ as a $j$-relative left adjoint for $H$, and so we deduce by \ref{J_rel_adj} that $\eta:1_{\C} \Rightarrow GF = HKF$ exhibits $KF:\C \to \B$ as left adjoint to $H:\B \to \C$.
\end{proof}

\begin{defn_sub}
\label{Jary_monadic}
{
Let $G : \A \to \C$ be a $\V$-functor, so that $(\A, G)$ is a $\V$-category over $\C$. We say that $G$, or $(\A, G)$, is \textbf{strictly $\J$-ary monadic} \cite[Definition 4.10]{Pres} if there is a $\J$-ary $\V$-monad $\T$ on $\C$ \eqref{Jary} such that $\A \cong \T\Alg$ in $\V\CAT/\C$.
}
\end{defn_sub}

\begin{para_sub}\label{para:jary_monadic}
A strictly $\J$-ary monadic $\V$-functor $G : \A \to \C$ is $\J$-tractable by Lemma \ref{adjointtractable}, and $G$ creates $\Phi_{\underJ}$-colimits by \cite[Proposition 4.11]{Pres} and thus conditionally preserves them (in the sense defined just before Proposition \ref{JnervousJflat}). In particular, if $\C$ has $\Phi_{\underJ}$-colimits, then $\A$ has them and $G$ preserves them.
\end{para_sub}

\begin{theo_sub}
\label{JnervousJarythm}
Suppose that $\J \hookrightarrow \C$ is eleutheric, and let $G : \A \to \C$ be a $\V$-functor. Then $G$ is strictly $\J$-ary monadic \eqref{Jary_monadic} iff $G$ is strictly $\J$-algebraic \eqref{J_alg_Vcategory}.
\end{theo_sub}

\begin{proof}
If $G$ is strictly $\J$-algebraic, then we deduce from Theorem \ref{Jnervous_thm} that $G$ is strictly monadic, and the associated $\V$-monad on $\C$ is $\J$-nervous. Since $\J$ is eleutheric, the $\V$-category $\C$ has $\Phi_{\underJ}$-colimits, and such colimits are $\J$-flat. Thus, the $\J$-nervous $\V$-monad $\T$ preserves $\Phi_{\underJ}$-colimits by Proposition \ref{JnervousJflat} (and the amenability of $\J$ established in Theorem \ref{eleuthericamenable}), so that $\T$ is $\J$-ary and $G$ is strictly $\J$-ary monadic. 

Suppose conversely that $G$ is strictly $\J$-ary monadic. Then the $\V$-functors $G : \A \to \C$ and $U^{\Str G} : \Str G\Alg^! \to \C$ are strictly monadic (the latter by Proposition  \ref{strictlymonadic} and Theorem \ref{eleuthericamenable}), and we have $U^{\Str G} \circ \sfE_{(\A, G)} = G$. We write $F$ for the left adjoint of $G$ and $\eta : 1 \Rightarrow GF = U^{\Str G} \sfE_{(\A, G)} F$ for the unit, so that $\eta j : j \Rightarrow GFj$ exhibits $Fj$ as a $j$-relative left adjoint for $G$ (by \ref{J_rel_adj}) and we may identify $\Str G$ with the $\J$-theory $\A_{Fj}^\op$ of Proposition \ref{Str_prop}. We deduce from the final assertion of Proposition \ref{Str_prop} that $\eta j : j \Rightarrow GFj = U^{\Str G} \sfE_{(\A, G)} Fj$ exhibits $\sfE_{(\A, G)} Fj$ as a $j$-relative left adjoint for $U^{\Str G}$. We now show that $\sfE_{(\A, G)} : \A \to \Str G\Alg^!$ is $\Phi_{\underJ}$-cocontinuous. Since the fully faithful $\sfM^{\Str G} : \Str G \Alg^! \rightarrowtail [\Str G, \V]$ reflects colimits, it suffices to show that $\sfM^{\Str G} \circ \sfE_{(\A, G)} = \widetilde{m}_G : \A \to [\Str G, \V]$ preserves $\Phi_{\underJ}$-colimits, for which it suffices to show that each $m_G J : \A \to \V$ ($J \in \ob\J$) preserves $\Phi_{\underJ}$-colimits (where $m_G : \Str G \to [\A, \V]$ is defined as in \ref{new_structure}). We have
\[ m_G J = m_G \tau_G J = \C(jJ, G-) = \C(J, G-), \] and $\C(J, G-)$ preserves $\Phi_{\underJ}$-colimits because $G$ preserves them \eqref{para:jary_monadic} and $\C(J, -)$ preserves them (since $\J$ is eleutheric). We finally deduce from Lemma \ref{monad_iso_lemma_2} that $\sfE_{(\A, G)} : \A \to \Str G\Alg^!$ is an isomorphism, and thus $G$ is strictly $\J$-algebraic.  
\end{proof}

\noindent We immediately obtain the following from Proposition \ref{Jnervous_equiv} and Theorems  \ref{eleuthericamenable} and \ref{JnervousJarythm}:

\begin{cor_sub}
\label{JnervousJarycor}
Suppose that $\J \hookrightarrow \C$ is eleutheric, and let $\T$ be a $\V$-monad on $\C$. Then $\T$ is $\J$-nervous iff $\T$ is $\J$-ary. \qed
\end{cor_sub}

\noindent From Theorems \ref{strsemadjunction}, \ref{monadtheoryequivalence}, and \ref{eleuthericamenable} we now deduce the following theorem, part of which (in view of Proposition \ref{systemaritiesthm}) generalizes the equivalence between $\J$-theories and $\J$-ary $\V$-monads on $\V$ for an eleutheric system of arities $j : \J \rightarrowtail \V$ \eqref{system_arities_para} previously established by the first author in \cite[Theorem 11.8]{EAT}.  

\begin{theo_sub}
\label{eleuthericequivalence}
Let $\J \hookrightarrow \C$ be an eleutheric subcategory of arities. Then we have a structure--semantics adjunction $\Str \dashv \Sem : \Th_{\underJ}(\C)^\op \to \J\Tract(\C)$ between $\J$-theories and $\J$-tractable $\V$-categories over $\C$, as well as an equivalence $\Mnd_{\underJ}(\C) \simeq \Th_{\underJ}(\C)$
between $\J$-theories and $\J$-nervous $\V$-monads on $\C$, equivalently $\J$-ary $\V$-monads on $\C$ \eqref{JnervousJarycor}. \qed 
\end{theo_sub}

\begin{egg_sub}
\label{eleutheric_str_sem}
{
The structure--semantics adjunction of Theorem \ref{eleuthericequivalence} specializes to recover the following structure--semantics adjunctions previously established in the literature (where each subcategory of arities is eleutheric by Example \ref{eleuthericexamples}):
\begin{enumerate}
[leftmargin=0pt,labelindent=0pt,itemindent=*,label=(\arabic*)]
\item By taking $\C = \V = \Set$ and $\J = \FinCard \hookrightarrow \Set$, we recover the structure--semantics adjunction between ordinary Lawvere theories and tractable categories over $\Set$ established by Lawvere in \cite{Law:PhD}.

\item By taking $\C$ to be an arbitrary $\V$-category enriched in an arbitrary symmetric monoidal closed category $\V$ with equalizers and $\J = \C \hookrightarrow \C$, we recover the structure--semantics adjunction between $\C$-theories and \emph{strongly tractable} $\V$-categories over $\C$ established by Dubuc in \cite{Dubucsemantics}. In the latter paper, it was also assumed that $\V$ is complete and well-powered. Given the equivalence between $\C$-theories and arbitrary $\V$-monads on $\C$ (see Example \ref{eleuthericexamples}), we also recover Dubuc's structure--semantics adjunction between $\V$-monads on $\C$ and strongly tractable $\V$-categories over $\C$ established in \cite[Page 74]{Dubucbook}. 

\item By taking $\C = \V$ a symmetric monoidal closed \emph{$\pi$-category} \cite{BorceuxDay} (e.g.~a complete and cocomplete cartesian closed category) and $\J = \SF(\V) \rightarrowtail \V$ the eleutheric system of arities consisting of the natural numbers, we recover the structure--semantics adjunction between $\SF(\V)$-theories and $\SF(\V)$-tractable $\V$-categories over $\V$ established by Borceux and Day in \cite[Theorem 2.4.2]{BorceuxDay}. 
\end{enumerate}
}
\end{egg_sub}

\subsection{Bounded and eleutheric subcategories of arities}
\label{boundedeleuthericsubsection}

In this subsection, we shall show that if $\J \hookrightarrow \C$ is an eleutheric subcategory of arities that satisfies the further condition of \emph{boundedness} \eqref{boundedVfunctor}, then not only is $\J$ itself strongly amenable, but every subcategory of arities that is \emph{contained} in $\J$ is strongly amenable (see Theorem \ref{boundedeleuthericamenable}). 

\begin{para_sub}
\label{boundedassumptions}
{
Throughout \S \ref{boundedeleuthericsubsection}, we assume that our symmetric monoidal closed $\V$ is complete and cocomplete, and that $\V$ is a \emph{closed factegory} \cite[Definition 6.1.2]{Pres}, meaning that $\V$ is equipped with an enriched (not necessarily proper) factorization system $(\E, \M)$ \cite{enrichedfact}. A \emph{$\V$-factegory} \cite[Definition 6.1.5]{Pres} is then a $\V$-category $\X$ equipped with an enriched factorization system $(\E_\X, \M_\X)$\footnote{We shall usually omit these subscripts when that is unlikely to cause confusion.} that is \emph{compatible} with $(\E, \M)$, meaning that each $\X(X, -) : \X \to \V$ ($X \in \ob\X$) sends $\M_\X$-morphisms to $\M$-morphisms. A $\V$-factegory $\X$ is \emph{proper} when $(\E_\X, \M_\X)$ is proper in the enriched sense, meaning that $\E_\X$ is contained in the $\V$-epimorphisms and $\M_\X$ in the $\V$-monomorphisms \cite[Definition 2.2]{enrichedfact}. If $\X$ is tensored and cotensored, then $(\E_\X, \M_\X)$ is proper in the enriched sense iff it is proper in the ordinary sense \cite[Proposition 2.4]{enrichedfact}. A \emph{cocomplete (proper) $\V$-factegory} \cite[Definition 6.1.5]{Pres} is a (proper) $\V$-factegory $\X$ that is cocomplete as a $\V$-category and has conical cointersections (i.e.~wide pushouts) of arbitrary families of $\E_\X$-morphisms with common domain.    

For the remainder of \S \ref{boundedeleuthericsubsection}, we let $\C$ be a cotensored $\V$-category that is also a cocomplete $\V$-factegory, and we suppose either that the enriched factorization system $(\E, \M) = (\E_\C, \M_\C)$ on $\C$ is proper or that $\C$ is $\E$-cowellpowered (meaning that each $C \in \ob\C$ has just a small set of $\E$-quotients, i.e.~isomorphism classes of $\E$-morphisms with domain $C$). 
}
\end{para_sub}

\begin{defn_sub}
\label{boundedVfunctor}
{
Let $\alpha$ be a regular cardinal, and let $\X$ and $\Y$ be $\V$-factegories with small conical $\alpha$-filtered colimits. A $\V$-functor $F : \X \to \Y$ is \textbf{$\alpha$-bounded} provided that for every small $\alpha$-filtered diagram $D : \A \to \X_0$ and every $\M$\emph{-cocone} $m = \left(m_A : DA \to X\right)_{A \in \A}$ on $D$ (meaning that $m$ is a cocone on $D$ such that each $m_A \in \M$), if the induced morphism $\overline{m} : \colim D \to X$ lies in $\E$, then the induced morphism $\overline{Fm} : \colim FD \to FX$ lies in $\E$. Thus, we say that $F$ is $\alpha$-bounded iff $F : \X \to \Y$ \emph{preserves the} $\E$\emph{-tightness of small} $\alpha$\emph{-filtered} $\M$\emph{-cocones}.\footnote{This terminology comes from Kelly \cite[2.3]{Kellytrans}.} If $(\E_\X, \M_\X)$ and $(\E_\Y, \M_\Y)$ are both the trivial factorization system $(\Iso, \All)$, then $F$ is $\alpha$-bounded iff $F$ preserves small conical $\alpha$-filtered colimits. 

An object $C$ of $\C$ is \textbf{$\alpha$-bounded} if the $\V$-functor $\C(C, -) : \C \to \V$ is $\alpha$-bounded. A subcategory of arities $\J \hookrightarrow \C$ is \textbf{$\alpha$-bounded} if $\J$ is small and each $J \in \ob\J$ is $\alpha$-bounded as an object of $\C$, and $\J \hookrightarrow \C$ is \textbf{bounded} if it is $\alpha$-bounded for some regular cardinal $\alpha$.  
}
\end{defn_sub}

\begin{egg_sub}
\label{boundedexamples}
{
We have the following examples of bounded subcategories of arities from \cite[Example 6.1.12]{Pres}:
\begin{enumerate}
[leftmargin=0pt,labelindent=0pt,itemindent=*,label=(\arabic*)]
\item Every small subcategory of arities in a \textit{locally bounded $\V$-category} $\C$ \cite{locbd} enriched in a \textit{locally bounded} closed category $\V$ \cite[Chapter 6]{Kelly} is bounded, with respect to the specified proper factorization system that $\C$ is required to carry as a locally bounded $\V$-category. In particular (by \cite[Example 5.14 and Proposition 7.1]{locbd}), every small subcategory of arities $\J \hookrightarrow \C$ in a locally presentable $\V$-category $\C$ enriched in a locally presentable closed category $\V$ is bounded.  

\item If small (conical) filtered colimits commute with (conical) finite powers in $\V$, i.e.~if each $(-)^n : \V \to \V$ ($n \in \N$) preserves small filtered colimits, then the system of arities $j : \SF(\V) \rightarrowtail \V$ consisting of the natural numbers \eqref{eleuthericexamples} is $\aleph_0$-bounded with respect to the trivial factorization system $(\Iso, \All)$, because $\V(n \cdot I, -) \cong (-)^n$ for each $n \in \N$. In fact, because the diagonal functor $\A \to \A \times \A$ for a small filtered category $\A$ is final, it suffices that each $X \times (-) : \V \to \V$ ($X \in \ob\V$) preserve small filtered colimits. In particular, this is true when $\V$ is cartesian closed, and more generally when $\V$ is a $\pi$-category \cite{BorceuxDay}. It is also true when $\V_0$ is locally finitely presentable, by \cite[3.8]{Kellystr}. Furthermore, this is also true for the symmetric monoidal closed category $\mathsf{Met}$ of extended metric spaces and non-expanding mappings by \cite[Corollary 2.4]{Metricmonads}, even though $\mathsf{Met}$ is not locally finitely presentable (although it is locally \emph{$\aleph_1$-presentable} by \cite[Examples 4.5(3)]{Metricclasses}) and is not known to be a $\pi$-category.  

\item The system of arities $\{I\} \hookrightarrow \V$ containing just the unit object is $\aleph_0$-bounded with respect to the trivial factorization system $(\Iso, \All)$. 

\item For a small $\V$-category $\A$, the subcategory of arities $\y_\A : \A^\op \rightarrowtail [\A, \V]$ is $\aleph_0$-bounded with respect to the trivial factorization system $(\Iso, \All)$.      
  
\item As in Example \ref{eleuthericexamples}, let $\Phi$ be a locally small class of small weights satisfying Axiom A of \cite{LR}, and let $\C = \Phi\Mod(\calT)$ be the $\V$-category of models of a $\Phi$-theory $\calT$. If $\V$ is a locally bounded closed category that is $\E$-cowellpowered, then $\y_\Phi : \calT^\op \rightarrowtail \Phi\Mod(\calT)$ is a bounded subcategory of arities with respect to a suitable proper factorization system carried by $\Phi\Mod(\calT)$.

\item Let $\Phi_\bbD$ be the class of small weights determined by a sound doctrine $\bbD$ as in Example \ref{eleuthericexamples}(7), where $\V$ is locally $\bbD$-presentable as a $\tensor$-category. Since $\bbD$ is a small class of small categories, we can find a regular cardinal $\alpha$ such that each $\D \in \bbD$ is $\alpha$-small. It then follows that $\y_{\Phi_{\bbD}} : \calT^\op \rightarrowtail \Phi_{\bbD}\Mod(\calT)$ is an $\alpha$-bounded subcategory of arities with respect to the trivial factorization system $(\Iso, \All)$.   
\end{enumerate}
}
\end{egg_sub}

In \cite[Theorem 9.1]{EP}, results pertaining to presentations of $\J$-ary $\V$-monads were used to prove that if $\J \hookrightarrow \C$ is contained in some bounded and eleutheric subcategory of arities, then $U^\scrT : \scrT\Alg^! \to \C$ is strictly monadic for every $\J$-pretheory $\scrT$. From this we obtain the following result: 

\begin{theo_sub}
\label{boundedeleuthericamenable}
Let $\J \hookrightarrow \C$ be a subcategory of arities that is contained in some bounded and eleutheric subcategory of arities. Then $\J \hookrightarrow \C$ is strongly amenable. \qed
\end{theo_sub}

\begin{egg_sub}
\label{boundedeleuthericexamples}
{
Every small subcategory of arities in a locally presentable $\V$-category $\C$ enriched in a locally presentable closed category $\V$ is contained in some bounded and eleutheric subcategory of arities by Examples \ref{eleuthericexamples}(1) and \ref{boundedexamples}(1). By Example  \ref{boundedexamples}(1), every small and eleutheric subcategory of arities in a locally bounded $\V$-category enriched in a locally bounded closed category $\V$ satisfies the hypotheses of Theorem \ref{boundedeleuthericamenable}. In view of Example \ref{boundedexamples}, the examples (2), (3), (5), and (7) of Example  \ref{eleuthericexamples} also satisfy the hypotheses of Theorem \ref{boundedeleuthericamenable}.   
}
\end{egg_sub}

\noindent From Theorems  \ref{strsemadjunction}, \ref{strsemcorestricted},  \ref{monadtheoryequivalence}, and \ref{boundedeleuthericamenable} we immediately obtain the following result:

\begin{theo_sub}
\label{boundedeleuthericthm}
Let $\J \hookrightarrow \C$ be a subcategory of arities that is contained in some bounded and eleutheric subcategory of arities. Then we have an idempotent structure--semantics adjunction
$\Str \dashv \Sem : \Preth_{\underJ}(\C)^\op \to \J\Tract(\C)$ between $\J$-pretheories and $\J$-tractable $\V$-categories over $\C$, from which we obtain an idempotent adjunction $\sfm \dashv \sft : \Mnd(\C) \to \Preth_{\underJ}(\C)$ between $\J$-pretheories and $\V$-monads on $\C$. This adjunction yields an equivalence $\Mnd_{\underJ}(\C) \simeq \Th_{\underJ}(\C)$ between $\J$-theories and $\J$-nervous $\V$-monads on $\C$ (which coincide with $\J$-ary $\V$-monads on $\C$ if $\J$ itself is eleutheric, \ref{JnervousJarycor}). \qed
\end{theo_sub}

\begin{rmk_sub}
\label{BGboundedeleutheric}
{
One of the central results of Bourke and Garner \cite[Theorem 6]{BourkeGarner} may be formulated in the new axiomatics of the present paper as the result that every small subcategory of arities in a locally presentable $\V$-category $\C$ enriched in a locally presentable closed category $\V$ is strongly amenable. In view of Example \ref{boundedeleuthericexamples}, we may now recover this result from Theorem \ref{boundedeleuthericamenable}. 
}
\end{rmk_sub}

\subsection{Small subcategories of arities in the locally bounded sketchable context}
\label{locallyboundedsubsection}

In this subsection, we shall show that if $\V$ is a locally bounded closed category, then every small subcategory of arities $\J \hookrightarrow \C$ in a \emph{$\V$-sketchable} $\V$-category $\C$ \eqref{locallyboundedassumptions} is strongly amenable.

\begin{para_sub}
\label{locallybounded}
{
We assume throughout this subsection that $\V$ is locally bounded as a closed category; see \cite[Chapter 6]{Kelly} and \cite[Definition 5.1]{locbd} for two equivalent formulations of the definition, which we shall not need. Note in particular that $\V$ is complete and cocomplete, and has an enriched proper factorization system $(\E, \M)$ \eqref{boundedassumptions}. We have the following examples of locally bounded closed categories from \cite{locbd}:
\begin{itemize}[leftmargin=*]
\item Kelly showed in \cite[Chapter 6]{Kelly} that locally bounded closed categories include all locally presentable closed categories and the cartesian closed categories $\mathsf{CGTop}$, $\mathsf{CGHTop}$, and $\mathsf{QTop}$ of compactly generated topological spaces, compactly generated Hausdorff spaces, and quasitopological spaces.
\item Every commutative unital quantale (i.e.~a symmetric monoidal closed category that is posetal and cocomplete) is a locally bounded closed category \cite[Example 5.15]{locbd}.

\item Given that $\V$ is locally bounded as a closed category, it is shown in \cite[Theorem 5.6]{KellyLack} that $\V\Cat$ is locally bounded as a closed category.  

\item Any cocomplete locally cartesian closed category with a (small) generator and arbitrary cointersections of epimorphisms is a locally bounded cartesian closed category \cite[Corollary 5.11]{locbd}. 

\item Specializing the previous example, the concrete quasitoposes of Dubuc \cite{Dubucquasitopoi}, and in particular the convenient categories of smooth spaces of Baez--Hoffnung \cite{Baezsmooth}, are locally bounded cartesian closed categories \cite[Example 5.17]{locbd}. These categories are not in general locally presentable, and include such examples as the categories of bornological sets, pseudotopological spaces, and various categories of convergence spaces \cite{Dubucquasitopoi}.

\item Generalizing the previous example, if $\V_0$ is a topological category over $\Set$, then the closed category $\V$ is locally bounded \cite[Proposition 5.13(1)]{locbd}. In particular, every cartesian closed topological category over $\Set$ is a locally bounded cartesian closed category. For example, given a \emph{productive class} $\mathcal{C}$ of topological spaces \cite{EscardoCCC}, the full subcategory $\mathsf{Top}_{\mathcal{C}}$ of $\mathsf{Top}$ consisting of the \emph{$\mathcal{C}$-generated spaces} is a locally bounded cartesian closed category \cite[Example 5.19]{locbd}. Examples of $\mathsf{Top}_{\mathcal{C}}$ include the categories of compactly generated spaces, core compactly generated spaces, locally compactly generated spaces, and sequentially generated spaces \cite[3.3]{EscardoCCC}. 

Moreover, every topological category over $\Set$ carries a canonical (generally non-cartesian) symmetric monoidal closed structure \cite[\S 2.2]{Sato} with respect to which it is a locally bounded closed category \cite[Example 5.18]{locbd}. For example, $\mathsf{Top}$ and $\mathsf{Meas}$ (the category of measurable spaces and measurable maps) become locally bounded closed categories with respect to this (non-cartesian) symmetric monoidal closed structure. 
\end{itemize}
}
\end{para_sub}  

\begin{para_sub}
\label{locallyboundedassumptions}
{
For the remainder of \S\ref{locallyboundedsubsection}, we fix a class $\Phi$ of small weights and a small $\Phi$-theory $\mathfrak{T}$, i.e.~a small $\Phi$-complete $\V$-category. We have the full sub-$\V$-category $\Phi\Cts(\mathfrak{T}, \V) \hookrightarrow [\mathfrak{T}, \V]$ consisting of the $\mathfrak{T}$-models in $\V$, i.e.~the $\Phi$-continuous $\V$-functors $\mathfrak{T} \to \V$. We assume for the remainder of \S\ref{locallyboundedsubsection} that the $\V$-category $\C$ is equivalent to $\Phi\Cts(\mathfrak{T}, \V)$ for some small $\Phi$-theory $\mathfrak{T}$, which we express by saying that $\C$ is a \textbf{$\V$-sketchable} $\V$-category; we shall assume without loss of generality that $\C = \Phi\Cts(\mathfrak{T}, \V)$.    
}
\end{para_sub}

\begin{para_sub}
\label{sketchablerepresentables}
{
Since $\C = \Phi\Cts(\mathfrak{T}, \V)$ contains the representables, the Yoneda embedding $\y : \mathfrak{T}^\op \rightarrowtail [\mathfrak{T}, \V]$ factors through $\Phi\Cts(\mathfrak{T}, \V)$ by way of a fully faithful $\V$-functor $\y_\Phi : \mathfrak{T}^\op \rightarrowtail \Phi\Cts(\mathfrak{T}, \V)$, which is dense by \cite[Theorem 5.13]{Kelly} because $\y$ is dense. If $\J \hookrightarrow \C = \Phi\Cts(\mathfrak{T}, \V)$ is a small full sub-$\V$-category that contains the representables $\mathfrak{T}(T, -) : \mathfrak{T} \to \V$ ($T \in \ob\mathfrak{T}$), then the dense and fully faithful $\y_\Phi : \mathfrak{T}^\op \rightarrowtail \Phi\Cts(\mathfrak{T}, \V)$ factors through $j : \J \hookrightarrow \Phi\Cts(\mathfrak{T}, \V)$ by way of a fully faithful $\V$-functor $k : \mathfrak{T}^\op \rightarrowtail \J$, so that $j$ and $k$ are both dense by \cite[Theorem 5.13]{Kelly}. In this case, we therefore have that $\J \hookrightarrow \Phi\Cts(\mathfrak{T}, \V)$ is a small subcategory of arities.
}
\end{para_sub}

\begin{lem_sub}
\label{jnervelem}
Let $j : \J \hookrightarrow \Phi\Cts(\mathfrak{T}, \V)$ be a small full sub-$\V$-category that contains the representables $\mathfrak{T}(T, -) : \mathfrak{T} \to \V$ ($T \in \ob\mathfrak{T}$). Then a presheaf $X : \J^\op \to \V$ is a $j$-nerve iff (i) $X$ is a right Kan extension of its restriction along $k^\op : \mathfrak{T} \rightarrowtail \J^\op$ \eqref{sketchablerepresentables} and (ii) $X \circ k^\op : \mathfrak{T} \to \V$ is $\Phi$-continuous.  
\end{lem_sub}

\begin{proof}
The $\V$-functor $\left[k^\op, 1\right] : [\J^\op, \V] \to [\mathfrak{T}, \V]$ has right adjoint $\Ran_{k^\op} : [\mathfrak{T}, \V] \to [\J^\op, \V]$, which is fully faithful because $k$ is fully faithful. Since $\V$ is locally bounded, we know by \cite[Theorem 11.8]{locbd} that $\Phi\Cts(\mathfrak{T}, \V)$ is reflective in $[\mathfrak{T}, \V]$. So the composite \[ \Phi\Cts(\mathfrak{T}, \V) \hookrightarrow [\mathfrak{T}, \V] \xrightarrow{\Ran_{k^\op}} [\J^\op, \V] \] is a fully faithful right adjoint, which is in fact isomorphic to $N_j : \Phi\Cts(\mathfrak{T}, \V) \to \left[\J^\op, \V\right]$, as we now show. Each $J \in \ob\J \subseteq \ob\left(\Phi\Cts(\mathfrak{T}, \V)\right)$ is a $\V$-functor $J : \mathfrak{T} \to \V$ and so by the Yoneda lemma is isomorphic to $[\mathfrak{T}, \V](\y-, J) = \J(k-, J)$. Invoking \cite[(3.7) and (4.9)]{Kelly}, we then have
\[ \left(\Ran_{k^\op} M\right)J = \left\{\J\left(k-, J\right), M\right\} \cong \{J, M\} \cong [\mathfrak{T}, \V](J, M) = \Phi\Cts(\mathfrak{T}, \V)(J, M) = \left(N_j M\right)J \] 
$\V$-naturally in $M \in \Phi\Cts(\mathfrak{T}, \V)$ and $J \in \J^\op$. So a presheaf $X : \J^\op \to \V$ is a $j$-nerve (i.e.~lies in the essential image of $N_j$) iff $X$ lies in the essential image of the reflective embedding $\Phi\Cts(\mathfrak{T}, \V) \hookrightarrow [\mathfrak{T}, \V] \xrightarrow{\Ran_{k^\op}} \left[\J^\op, \V\right]$, which is equivalent to $X$ satisfying (i) and (ii).  
\end{proof}

\begin{para_sub}
\label{density_pres}
{
Let $K : \A \to \B$ be a $\V$-functor. Recall from \cite[\S 5.4]{Kelly} that a weighted colimit in $\B$ is \emph{$K$-absolute} if it is preserved by the nerve $\V'$-functor $N_K : \B \to \left[\A^\op, \V\right]$. A \emph{density presentation} for a fully faithful $\V$-functor $K : \A \rightarrowtail \B$ is then a class $\Psi$ of weighted diagrams in $\B$ such that each diagram in $\Psi$ has a $K$-absolute colimit and $\B$ is the closure of $\A$ under $\Psi$-colimits. By \cite[Theorem 5.19]{Kelly}, a fully faithful $\V$-functor $K : \A \rightarrowtail \B$ is dense iff it has a density presentation, and when $\A$ is small we can (and shall) assume that the density presentation consists of \emph{small} weighted diagrams.      
}
\end{para_sub}

\begin{prop_sub}
\label{jnerveprop}
Let $j : \J \hookrightarrow \Phi\Cts(\mathfrak{T}, \V)$ be a small full sub-$\V$-category that contains the representables $\mathfrak{T}(T, -) : \mathfrak{T} \to \V$ ($T \in \ob\mathfrak{T}$), and let $\Psi$ be a density presentation for the fully faithful and dense $\V$-functor $k : \mathfrak{T}^\op \rightarrowtail \J$ \eqref{sketchablerepresentables}.  Then a presheaf $X : \J^\op \to \V$ is a $j$-nerve iff $X$ preserves $\Psi$-limits and $X \circ k^\op : \mathfrak{T} \to \V$ is $\Phi$-continuous. 

Consequently, for a $\J$-pretheory $\scrT$, a $\V$-functor $M : \scrT \to \V$ is a (non-concrete) $\scrT$-algebra iff $M \circ \tau : \J^\op \to \V$ preserves $\Psi$-limits and $M \circ \tau \circ k^\op : \mathfrak{T} \to \V$ is $\Phi$-continuous.  
\end{prop_sub}

\begin{proof}
The second assertion follows from the first because a $\V$-functor $M : \scrT \to \V$ is a $\scrT$-algebra iff $M \circ \tau : \J^\op \to \V$ is a $j$-nerve. We know from Lemma \ref{jnervelem} that a presheaf $X : \J^\op \to \V$ is a $j$-nerve iff $X \cong \Ran_{k^\op} \left(X \circ k^\op\right)$ and $X \circ k^\op$ is $\Phi$-continuous. But $X \cong \Ran_{k^\op} \left(X \circ k^\op\right)$ iff $X$ preserves $\Psi$-limits by \cite[Theorem 5.29]{Kelly}. 
\end{proof}

\begin{para_sub}
\label{sketchpara}
{
Let $\J \hookrightarrow \Phi\Cts(\mathfrak{T}, \V)$ be a small full sub-$\V$-category that contains the representables $\mathfrak{T}(T, -) : \mathfrak{T} \to \V$ ($T \in \ob\mathfrak{T}$), and let $\Psi$ be a density presentation for the fully faithful and dense $\V$-functor $k : \mathfrak{T}^\op \rightarrowtail \J$ \eqref{sketchablerepresentables}. Given a $\J$-pretheory $\scrT$, we can equip the small $\V$-category $\scrT$ with the structure of an \emph{enriched limit sketch} $(\scrT, \Gamma)$ (see \cite[\S 6.3]{Kelly} or \cite[\S 11]{locbd}) as follows. We equip $\scrT$ with the cylinders obtained by applying the $\V$-functor $\tau : \J^\op \to \scrT$ to the $\Psi$-limit cylinders in $\J^\op$, and with the cylinders obtained by applying the $\V$-functor $\tau \circ k^\op : \mathfrak{T} \to \scrT$ to the $\Phi$-limit cylinders in $\mathfrak{T}$. A \emph{model} of the limit sketch $(\scrT, \Gamma)$ is then a $\V$-functor $M : \scrT \to \V$ that sends all of these cylinders to limit cylinders in $\V$, and we write $(\scrT, \Gamma)\Mod \hookrightarrow [\scrT, \V]$ for the full sub-$\V$-category consisting of the models of $(\scrT, \Gamma)$.
}
\end{para_sub}

\begin{prop_sub}
\label{sketchprop}
Let $\J \hookrightarrow \Phi\Cts(\mathfrak{T}, \V)$ be a small full sub-$\V$-category that contains the representables $\mathfrak{T}(T, -) : \mathfrak{T} \to \V$ ($T \in \ob\mathfrak{T}$), and let $\Psi$ be a density presentation for the fully faithful and dense $\V$-functor $k : \mathfrak{T}^\op \rightarrowtail \J$ \eqref{sketchablerepresentables}. Let $\scrT$ be a $\J$-pretheory, and let $(\scrT, \Gamma)$ be the associated limit sketch of \eqref{sketchpara}. Then $\scrT\Alg = (\scrT, \Gamma)\Mod$ as full sub-$\V$-categories of $[\scrT, \V]$. 
\end{prop_sub}

\begin{proof}
It is immediate that a $\V$-functor $M : \scrT \to \V$ is a model of the limit sketch $(\scrT, \Gamma)$ iff $M \circ \tau : \J^\op \to \V$ preserves $\Psi$-limits and $M \circ \tau \circ k^\op : \mathfrak{T} \to \V$ is $\Phi$-continuous. The result then immediately follows from Proposition \ref{jnerveprop}.
\end{proof}

\begin{prop_sub}
\label{locbdalgebras}
Let $\J \hookrightarrow \Phi\Cts(\mathfrak{T}, \V)$ be a small full sub-$\V$-category that contains the representables $\mathfrak{T}(T, -) : \mathfrak{T} \to \V$ ($T \in \ob\mathfrak{T}$), and let $\scrT$ be a $\J$-pretheory. Then $\scrT\Alg \hookrightarrow [\scrT, \V]$ is reflective. If $\V$ is $\E$-cowellpowered, then $\scrT\Alg$ is a locally bounded and $\E$-cowellpowered $\V$-category, and the inclusion $\scrT\Alg \hookrightarrow [\scrT, \V]$ is a bounding right adjoint \cite[Definition 4.33]{locbd}. 
\end{prop_sub}

\begin{proof}
The result follows immediately from Proposition \ref{sketchprop} and \cite[Theorem 11.5]{locbd}, noting that the first assertion also follows from \cite[Theorem 6.11]{Kelly} (in view of Proposition \ref{sketchprop}).  
\end{proof}

\begin{lem_sub}
\label{LB_subcat_lem}
Let $j : \J \hookrightarrow \Phi\Cts(\mathfrak{T}, \V)$ be a small full sub-$\V$-category that contains the representables $\mathfrak{T}(T, -) : \mathfrak{T} \to \V$ ($T \in \ob\mathfrak{T}$). Then $\J$ is a strongly amenable subcategory of arities. 
\end{lem_sub}

\begin{proof}
Let $\scrT$ be a $\J$-pretheory. Then $\scrT\Alg^!$ is a $\V$-category because $\J$ is small and $\V$ is complete. In view of \ref{concreteequivalence}, it suffices to show that the $\V$-functor $W^\scrT : \scrT\Alg \to j\Ner(\V)$ has a left adjoint. The composite $\V$-functor $\scrT\Alg \xrightarrow{W^\scrT} j\Ner(\V) \hookrightarrow \left[\J^\op, \V\right]$ factors as the composite $\scrT\Alg \hookrightarrow [\scrT, \V] \xrightarrow{\left[\tau, 1\right]} \left[\J^\op, \V\right]$. The inclusion $\scrT\Alg \hookrightarrow [\scrT, \V]$ has a left adjoint by Proposition \ref{locbdalgebras}, while $[\tau, 1]$ has a left adjoint given by left Kan extension along $\tau$ (since $\scrT$ is small and $\V$ is cocomplete). So the composite $\V$-functor $\scrT\Alg \xrightarrow{W^\scrT} j\Ner(\V) \hookrightarrow \left[\J^\op, \V\right]$ has a left adjoint, which restricts to a left adjoint for $W^\scrT$. 
\end{proof}

\noindent We can now prove our central results of this subsection:

\begin{theo_sub}
\label{locbdtractable}
Let $\V$ be a locally bounded closed category, and let $\C$ be a $\V$-sketchable $\V$-category \eqref{locallyboundedassumptions}.
\begin{enumerate}[leftmargin=*]
\item Every small subcategory of arities $\J \hookrightarrow \C$ is strongly amenable.
\item Every small full sub-$\V$-category $\J \hookrightarrow \C$ is contained in a small and strongly amenable subcategory of arities.
\end{enumerate}
\end{theo_sub}

\begin{proof}
Without loss of generality, $\C = \Phi\Cts(\mathfrak{T}, \V)$ for a small $\Phi$-theory $\mathfrak{T}$, by \ref{locallyboundedassumptions}. It suffices to prove (2), because (1) then follows by Theorem \ref{pushoutthm}. Given $\J \hookrightarrow \C$ as in (2), let us write $\scrK := \J \cup \left\{\mathfrak{T}(T, -) \mid T \in \ob\mathfrak{T}\right\} \hookrightarrow \C$. Then $\scrK$ is a small subcategory of arities \eqref{sketchablerepresentables} that contains $\J$ and is strongly amenable by Lemma \ref{LB_subcat_lem}. 
\end{proof}

\noindent The $\V$-category $\V$ is itself $\V$-sketchable when we take $\Phi$ to be  the empty class of weights and the small $\Phi$-theory (i.e.~$\V$-category) $\mathfrak{T}$ to be the unit $\V$-category. Any small full sub-$\V$-category of $\V$ that contains the unit object is automatically dense by \ref{sketchablerepresentables}, and hence is a small subcategory of arities. We then immediately obtain the following result from Lemma \ref{LB_subcat_lem} or Theorem \ref{locbdtractable}: 

\begin{theo_sub}
\label{locbdVcase}
Let $\V$ be a locally bounded closed category. Then every small full sub-$\V$-category $\J \hookrightarrow \V$ that contains the unit object is a strongly amenable subcategory of arities. \qed
\end{theo_sub}

\noindent From Theorems \ref{strsemadjunction}, \ref{strsemcorestricted},  \ref{monadtheoryequivalence}, and \ref{locbdtractable} we immediately obtain the following result:

\begin{theo_sub}
\label{locbdthm}
Let $\V$ be a locally bounded closed category, and let $\J \hookrightarrow \C$ be any small subcategory of arities in a $\V$-sketchable $\V$-category $\C$ \eqref{locallyboundedassumptions}. Then we have an idempotent structure--semantics adjunction
$\Str \dashv \Sem : \Preth_{\underJ}(\C)^\op \to \J\Tract(\C)$ between $\J$-pretheories and $\J$-tractable $\V$-categories over $\C$, from which we obtain an idempotent adjunction $\sfm \dashv \sft : \Mnd(\C) \to \Preth_{\underJ}(\C)$ between $\J$-pretheories and $\V$-monads on $\C$. This adjunction restricts to an equivalence $\Mnd_{\underJ}(\C) \simeq \Th_{\underJ}(\C)$ between $\J$-theories and $\J$-nervous $\V$-monads on $\C$. \qed
\end{theo_sub}

\noindent From Proposition \ref{locbdalgebras} and Theorem \ref{locbdtractable} (together with Theorem \ref{monadtheoryequivalence}) we also immediately deduce the following result about $\V$-categories of algebras for $\J$-nervous $\V$-monads:

\begin{prop_sub}
\label{Jnervouslocallybounded}
Let $\V$ be a locally bounded and $\E$-cowellpowered closed category, let $\J \hookrightarrow \C$ be any small subcategory of arities in a $\V$-sketchable $\V$-category $\C$ \eqref{locallyboundedassumptions}, and let $\T$ be a $\J$-nervous $\V$-monad on $\C$. Then $\T\Alg$ is a locally bounded and $\E$-cowellpowered $\V$-category, as is any $\J$-algebraic $\V$-category over $\C$. \qed
\end{prop_sub}

\begin{rmk_sub}
\label{BGlocbd}
{
In Remark \ref{BGboundedeleutheric} we indicated one way in which certain results of the present paper specialize to recover one of the main results of Bourke and Garner \cite{BourkeGarner}, which, when formulated in the new axiomatics of the present paper, is the statement that every small subcategory of arities in a locally presentable $\V$-category enriched in a locally presentable closed category $\V$ is strongly amenable. The results of this subsection provide yet another way to obtain this central result of \cite{BourkeGarner}, as that result follows from Theorem \ref{locbdtractable}, the fact that every locally presentable closed category is a locally bounded closed category \cite[Example 5.14]{locbd}, and the fact that every locally presentable $\V$-category enriched in a locally presentable closed category $\V$ is $\V$-sketchable \cite[7.4]{Kellystr}.    
}
\end{rmk_sub}

\bibliographystyle{amsplain}
\bibliography{mybib}

\newcommand{\noopsort}[1]{}
\providecommand{\bysame}{\leavevmode\hbox to3em{\hrulefill}\thinspace}
\providecommand{\MR}{\relax\ifhmode\unskip\space\fi MR }
\providecommand{\MRhref}[2]{%
  \href{http://www.ams.org/mathscinet-getitem?mr=#1}{#2}
}
\providecommand{\href}[2]{#2}
\begin{thebibliography}{10}

\bibitem{ABLR}
Ji\v{r}\'{\i} Ad\'{a}mek, Francis Borceux, Stephen Lack, and Ji\v{r}\'{\i}
  Rosick\'{y}, \emph{A classification of accessible categories}, J. Pure Appl.
  Algebra \textbf{175} (2002), no.~1-3, 7--30.

\bibitem{AHS}
Ji\v{r}\'{\i} Ad\'{a}mek, Horst Herrlich, and George~E. Strecker,
  \emph{Abstract and concrete categories: the joy of cats}, Repr. Theory Appl.
  Categ. (2006), no.~17, 1--507, Reprint of the 1990 original [Wiley, New
  York].

\bibitem{Arkor_thesis}
Nathanael Arkor, \emph{Monadic and higher-order structure}, Ph.D. thesis,
  University of Cambridge, 2022.

\bibitem{formal_relative_monads}
Nathanael Arkor and Dylan McDermott, \emph{The formal theory of relative
  monads}, J. Pure Appl. Algebra \textbf{228} (2024), no.~9, Paper No. 107676,
  107.

\bibitem{Avery_thesis}
Tom Avery, \emph{Structure and semantics}, Ph.D. thesis, University of
  Edinburgh, 2017.

\bibitem{Baezsmooth}
John~C. Baez and Alexander~E. Hoffnung, \emph{Convenient categories of smooth
  spaces}, Trans. Amer. Math. Soc. \textbf{363} (2011), no.~11, 5789--5825.

\bibitem{BMW}
Clemens Berger, Paul-Andr\'{e} Melli\`{e}s, and Mark Weber, \emph{Monads with
  arities and their associated theories}, J. Pure Appl. Algebra \textbf{216}
  (2012), no.~8-9, 2029--2048.

\bibitem{Birkhoff}
Garrett Birkhoff, \emph{On the structure of abstract algebras}, Proc. Camb.
  Phil. Soc \textbf{31} (1935), no.~31, 433--454.

\bibitem{Borceux1}
Francis Borceux, \emph{Handbook of categorical algebra 1}, Encyclopedia of
  Mathematics and its Applications, vol.~50, Cambridge University Press,
  Cambridge, 1994.

\bibitem{Borceux2}
\bysame, \emph{Handbook of categorical algebra 2}, Encyclopedia of Mathematics
  and its Applications, vol.~51, Cambridge University Press, Cambridge, 1994.

\bibitem{BorceuxDay}
Francis Borceux and Brian Day, \emph{Universal algebra in a closed category},
  J. Pure Appl. Algebra \textbf{16} (1980), no.~2, 133--147.

\bibitem{BourkeGarner}
John Bourke and Richard Garner, \emph{Monads and theories}, Adv. Math.
  \textbf{351} (2019), 1024--1071.

\bibitem{Carboni_localization}
A.~Carboni, G.~Janelidze, G.~M. Kelly, and R.~Par\'{e}, \emph{On localization
  and stabilization for factorization systems}, Appl. Categ. Structures
  \textbf{5} (1997), no.~1, 1--58.

\bibitem{Diers}
Yves Diers, \emph{Foncteur pleinement fid\'{e}le dense classant les
  alg\'{e}bres}, Cahiers Topologie G\'{e}om. Diff\'{e}rentielle Cat\'{e}g.
  \textbf{17} (1976), no.~2, 171--186.

\bibitem{Dubucbook}
Eduardo~J. Dubuc, \emph{Kan extensions in enriched category theory}, Lecture
  Notes in Mathematics, Vol. 145, Springer-Verlag, Berlin-New York, 1970.

\bibitem{Dubucsemantics}
\bysame, \emph{Enriched semantics-structure (meta) adjointness}, Rev. Un. Mat.
  Argentina \textbf{25} (1970/71), 5--26.

\bibitem{Dubucquasitopoi}
\bysame, \emph{Concrete quasitopoi}, Applications of sheaves, Lecture Notes in
  Math., vol. 753, Springer, Berlin, 1979, pp.~239--254.

\bibitem{EscardoCCC}
Mart\'{\i}n Escard\'{o}, Jimmie Lawson, and Alex Simpson, \emph{Comparing
  {C}artesian closed categories of (core) compactly generated spaces}, Topology
  Appl. \textbf{143} (2004), no.~1-3, 105--145.

\bibitem{Fujii_framework}
Soichiro Fujii, \emph{A unified framework for notions of algebraic theory},
  Theory Appl. Categ. \textbf{34} (2019), 1246--1316.

\bibitem{Kellytrans}
G.~M. Kelly, \emph{A unified treatment of transfinite constructions for free
  algebras, free monoids, colimits, associated sheaves, and so on}, Bull.
  Austral. Math. Soc. \textbf{22} (1980), no.~1, 1--83.

\bibitem{Kellystr}
\bysame, \emph{Structures defined by finite limits in the enriched context
  {I}}, Cahiers Topologie G\'{e}om. Diff\'{e}rentielle Cat\'{e}g. \textbf{23}
  (1982), no.~1, 3--42.

\bibitem{Kelly}
\bysame, \emph{Basic concepts of enriched category theory}, Repr. Theory Appl.
  Categ. (2005), no.~10, Reprint of the 1982 original [Cambridge Univ. Press,
  Cambridge].

\bibitem{KellyLackstronglyfinitary}
G.~M. Kelly and Stephen Lack, \emph{Finite-product-preserving functors, {K}an
  extensions and strongly-finitary {$2$}-monads}, Appl. Categ. Structures
  \textbf{1} (1993), no.~1, 85--94.

\bibitem{KellyLack}
\bysame, \emph{{$\mathscr{V}$}-{C}at is locally presentable or locally bounded
  if {$\mathscr{V}$} is so}, Theory Appl. Categ. \textbf{8} (2001), 555--575.

\bibitem{KS}
G.~M. Kelly and V.~Schmitt, \emph{Notes on enriched categories with colimits of
  some class}, Theory Appl. Categ. \textbf{14} (2005), no. 17, 399--423.

\bibitem{LR}
Stephen Lack and Ji\v{r}\'{\i} Rosick\'{y}, \emph{Notions of {L}awvere theory},
  Appl. Categ. Structures \textbf{19} (2011), no.~1, 363--391.

\bibitem{Law:PhD}
F.~W. Lawvere, \emph{Functorial semantics of algebraic theories}, Ph.D. thesis,
  Columbia University, New York, 1963, Available in: \emph{Repr. Theory Appl.
  Categ.} \textbf{5} (2004).

\bibitem{Metricclasses}
M.~Lieberman and J.~Rosick\'{y}, \emph{Metric abstract elementary classes as
  accessible categories}, J. Symb. Log. \textbf{82} (2017), no.~3, 1022--1040.

\bibitem{Lintonequational}
F.~E.~J. Linton, \emph{Some aspects of equational categories}, Proc. {C}onf.
  {C}ategorical {A}lgebra ({L}a {J}olla, {C}alif., 1965), Springer, New York,
  1966, pp.~84--94.

\bibitem{Lintonoutline}
\bysame, \emph{An outline of functorial semantics}, Sem. on {T}riples and
  {C}ategorical {H}omology {T}heory ({ETH}, {Z}\"{u}rich, 1966/67), Springer,
  Berlin, 1969, pp.~7--52.

\bibitem{enrichedfact}
Rory B.~B. Lucyshyn-Wright, \emph{Enriched factorization systems}, Theory Appl.
  Categ. \textbf{29} (2014), No. 18, 475--495.

\bibitem{EAT}
\bysame, \emph{Enriched algebraic theories and monads for a system of arities},
  Theory Appl. Categ. \textbf{31} (2016), No. 5, 101--137.

\bibitem{locbd}
Rory B.~B. Lucyshyn-Wright and Jason Parker, \emph{Locally bounded enriched
  categories}, Theory Appl. Categ. \textbf{38} (2022), No. 18, 684--736.

\bibitem{Pres}
\bysame, \emph{Presentations and algebraic colimits of enriched monads for a
  subcategory of arities}, Theory Appl. Categ. \textbf{38} (2022), No. 38,
  1434--1484.

\bibitem{EP}
\bysame, \emph{Diagrammatic presentations of enriched monads and varieties for
  a subcategory of arities}, Appl. Categ. Structures \textbf{31} (2023), no.~5,
  Paper No. 40, 39.

\bibitem{Ax}
\bysame, \emph{Equational presentations and the axiomatics of enriched algebra
  for a subcategory of arities}, In preparation, 2023.

\bibitem{NishizawaPower}
Koki Nishizawa and John Power, \emph{Lawvere theories enriched over a general
  base}, J. Pure Appl. Algebra \textbf{213} (2009), no.~3, 377--386.

\bibitem{PowerLawvere}
John Power, \emph{Enriched {L}awvere theories}, Theory Appl. Categ. \textbf{6}
  (1999), 83--93.

\bibitem{Metricmonads}
Ji\v{r}\'{\i} Rosick\'{y}, \emph{Metric monads}, Math. Structures Comput. Sci.
  \textbf{31} (2021), no.~5, 535--552.

\bibitem{Sato}
Tetsuya Sato, \emph{The {G}iry monad is not strong for the canonical symmetric
  monoidal closed structure on {$\bold{Meas}$}}, J. Pure Appl. Algebra
  \textbf{222} (2018), no.~10, 2888--2896.

\bibitem{Street_monads}
Ross Street, \emph{The formal theory of monads}, J. Pure Appl. Algebra
  \textbf{2} (1972), no.~2, 149--168.

\bibitem{Webernerve}
Mark Weber, \emph{Familial 2-functors and parametric right adjoints}, Theory
  Appl. Categ. \textbf{18} (2007), No. 22, 665--732.

\bibitem{WolffVcat}
Harvey Wolff, \emph{{$V$}-cat and {$V$}-graph}, J. Pure Appl. Algebra
  \textbf{4} (1974), 123--135.

\end{thebibliography}

\end{document}